\documentclass[a4paper]{article}
\usepackage{amsmath,amsthm,amssymb,bm, comment}
\usepackage{authblk}
\usepackage{xcolor}
\usepackage{graphicx} 
\usepackage{hyperref}
\usepackage{booktabs}
\usepackage{adjustbox}
\usepackage{multirow}
\usepackage[normalem]{ulem}
\usepackage{enumitem}

\usepackage[linesnumbered,ruled,vlined]{algorithm2e}
\SetKwRepeat{Do}{do}{while}
\SetKw{True}{\texttt{True}}
\SetKw{False}{\texttt{False}}

\usepackage[title]{appendix}

\DeclareMathOperator{\ord}{ord}

\newcommand{\bc}{\mathbf{c}}

\newcommand{\cA}{\mathcal{A}}
\newcommand{\cB}{\mathcal{B}}
\newcommand{\cC}{\mathcal{C}}
\newcommand{\cD}{\mathcal{D}}
\newcommand{\cE}{\mathcal{E}}

\newcommand{\cG}{\mathcal{G}}
\newcommand{\cH}{\mathcal{H}}

\newcommand{\cO}{\mathcal{O}}

\newcommand{\cX}{\mathcal{X}}
\newcommand{\cM}{\mathcal{M}}

\newcommand{\bbN}{\mathbb{N}}
\newcommand{\bbZ}{\mathbb{Z}}
\newcommand{\bbC}{\mathbb{C}}

\newcommand{\bbK}{\mathbb{K}}
\newcommand{\bbL}{\mathbb{L}}
\newcommand{\bbP}{\mathbb{P}}
\newcommand{\bbQ}{\mathbb{Q}}

\newcommand{\coC}{\mathbf{C}}

\newcommand{\coeff}{\operatorname{coeff}}

\newcommand{\centr}{\mathcal{Z}}

\newcommand{\rank}{{\rm rank}}
\newcommand{\gd}{{\rm GD}}
\newcommand{\hgd}{{\rm hGD}}

\newtheorem{thm}{Theorem}[section]
\newtheorem{theorem}[thm]{Theorem}

\newtheorem{proposition}[thm]{Proposition}

\newtheorem{lemma}[thm]{Lemma}
\newtheorem{cor}{Corollary}[thm]

\theoremstyle{definition}

\newtheorem{definition}[thm]{Definition}

\newtheorem{example}[thm]{Example}

\theoremstyle{remark}
\newtheorem{rem}[thm]{Remark}

\SetKwInOut{Input}{Input}
\SetKwInOut{Output}{Output}
\SetKwComment{Comment}{// }{}

\title{Computing centralizer of ODOs}
\title{Effective computation of centralizers of ODOs}
\author[a]{Antonio Jiménez-Pastor}
\author[a]{Sonia L. Rueda}

\affil[a]{Universidad Politécnica de Madrid}
\date{}

\begin{document}

\maketitle

\color{black}
\begin{abstract}
This work is devoted to computing the centralizer $\centr (L)$ of an ordinary differential operator (ODO) in the ring of differential operators.
Non-trivial centralizers are known to be coordinate rings of spectral curves and contain the ring of polynomials $\coC [L]$, with coefficients in the field of constants $\coC$ of $L$.

We give an algorithm to compute a basis of $\centr (L)$ as a $\coC [L]$-module. Our approach combines results by K. Goodearl in 1985 with solving the systems of equations of the stationary Gelfand-Dickey hierarchy, which after substituting the coefficients of $L$ become linear, and whose solution sets form a flag of constants. We are assuming that the coefficients of $L$ belong to a differential algebraic extension $K$ of $\coC$.
In addition, by considering parametric coefficients we develop an algorithm to generate families of ODOs with non trivial centralizer, in particular algebro-geometric, whose coefficients are solutions in $K$ of systems of the stationary GD hierarchy.
\end{abstract}

 \textsc{\href{https://zbmath.org/classification/}{MSC[2020]}:}  13N10, 13P15, 12H05



\section{Introduction}\label{sec:introduction}

The centralizer $\centr (L)$ of an ordinary differential operator (ODO) with coefficients in a differential field $\Sigma$ is a maximal commutative ring and has been of importance to differential operator theorists and ring theorists since the beginning of the twentieth century \cite{Good}. The ring $\coC [L]$ of polynomials in $L$, with coefficients in the field of constants $\coC$ of $L$ is contained in $\centr (L)$. In most cases, $\centr (L)$ coincides with $\coC [L]$, but of special interest are the cases where $\coC [L]\subsetneq \centr (L)$, and in this situation we will say that the centralizer is nontrivial \cite{MRZ2}, \cite{RZ2024}. If in addition $\centr (L)$ contains two operators of coprime orders then the operator $L$ is classically called algebro-geometric
\cite{Gri}, \cite{We}.

The correspondence between commutative rings of ordinary differential operators  and algebraic curves has been extensively and deeply studied since the seminal work by Burchnall and Chaundy in 1923 \cite{BC}, \cite{GGKM}, \cite{Krichever}, \cite{Mu2}, \cite{Zheglov}. We now see the possibility of making this correspondence computationally effective using differential algebra and symbolic computation \cite{Previato2019}, \cite{MRZ1}, \cite{MRZ2}, \cite{RZ2024}.
This paper originates from the need to generate families of differential operators with non-trivial centralizers,  for ODOs of arbitrary order \cite{R2025arx}.
It is a non trivial problem to determine if a given ODO
$L$ has a trivial centralizer and there are no algorithms for this task. The efforts concentrated in generating pairs of commuting operators of arbitrary rank, the greatest common divisor of their orders, \cite{DGU}, \cite{Mi1}, \cite{Mokhov_1990}, \cite{Og2018},\cite{PZ}, but the computation of the whole centralizer was not considered.

\medskip

The ultimate purpose of this paper is to design algorithms that can be implemented in a computer algebra system, for two main goals:
\begin{enumerate}[label=(G\arabic*)]
    \item\label{goal:1} Compute a finite set of generators of the centralizer of $L$ in the ring $\cD$ of ordinary differential operators  with coefficients in $\Sigma$
\[\centr (L)=\{A\in \cD\mid [L,A]=0\}.\]

    \item\label{goal:2} Generate families of ODOs with non trivial centralizers for operators $L$ of arbitrary order.
\end{enumerate}

We combine results on $\centr (L)$ that were developed throughout the twentieth century and have a long history.
On one hand one can consider a larger centralizer $\centr ((L))$, the ring of pseudo-differential operators that commute with $L$, and use it to compute $\centr (L)$ since
\[\centr (L)=\centr ((L))\cap \cD.\]
A result first stated in 1904 by I. Schur  \cite{Sch}, evolved to say that for an operator of order $n$, the $n$-th root $Q$ of $L$ determines $\centr ((L))$, since the powers $\{Q^j\}$ form an infinite basis as a $\coC$-vector space \cite{Good}.
In 1985, G. Wilson \cite{Wilson}
defined the set of differential operators that almost commute with $L$. In fact this is the set $\centr((L))_+$ of all positive parts of the operators in $\centr ((L))$, the $\coC$-vector space with infinite basis $\{(Q^j)_+\}$ of almost commuting operators \cite{JRZHD2025}.

\medskip

For a formal differential operator
$L_n$ of order $n$ in normal form (the term of order $n-1$ is zero), whose coefficients are differential variables over $\coC$,
the computation of the basis of almost commuting operators and of the systems of differential polynomials of the Gelfand-Dickey (GD) hierarchies \cite{Dikii} (that include KdV, Boussinesq, KN,...) was algorithmically described in \cite{JRZHD2025} and implemented in the computer algebra system SAGE. For each positive integer $m$, from the commutator of $L_n$ with an almost commuting operator of order $m$, the system $\gd_{n,m}$ of the stationary GD hierarchy is obtained.

\medskip

On the other hand, it was proven by K. Goodearl in 1985 \cite{Good} that $\centr (L)$ has a finite basis $\{1, A_1,\ldots ,A_{d-1}\}$ as a $\coC [L]$-module. Consequently
\[\centr (L)=\coC [L, A_1,\ldots ,A_{d-1}].\]
We combine this result with solving
systems of equations of the stationary GD hierarchy to achieve \ref{goal:1}.
The departure assumptions are that the coefficients of $L$ belong to a differential algebraic extension $\coC \langle \eta \rangle$ of $\coC$ and that the level $M$ of $L$ is known, being $M$ the smallest order of the ODOs in $\centr (L)\backslash \coC [L]$.  After evaluating $\gd_{n,m}$ in the coefficients of $L$, these systems become systems of linear equations with coefficients in $\Sigma=\coC \langle \eta \rangle$, and their sets of constant solutions can be computed.

Considering parametric coefficients in a differential field extension $\bbC(\Theta)\langle \eta \rangle$ of the field of complex numbers $\bbC$, determined by a finite set of parameters $\Theta$, we compute families of algebro-geometric ODOs and achieve \ref{goal:2}. In other words, we compute families of solutions of the systems $\gd_{n,M}$ of the stationary GD hierarchy in $\coC\langle\eta\rangle$, fixing $\eta$  equal to $x$, $\cosh (x)$ and the Weierstrass $\wp$ function, with derivation $d/dx$. For any fixed values of $n$ and $M$, we design an algorithm to describe the so called $M$-th level variety, as a zero set of a polynomial ideal in $\coC [\Theta]$. The problem now becomes solving a nonlinear system of polynomial equations in $\coC[\Theta]$. Our experimental results include  families of algebro-geometric operators of orders $n=3$ and $n=5$, furthermore they allow to detect new theoretical and computational challanges related with the generation of solutions of $\gd_{n,M}$ that we leave as a conjecture in Section \ref{sec-levelvar}: If $(n,M)=1$ finite families of algebro-geometric ODOs can be computed, otherwise infinite families of ODOs with centralizer of rank $(n,M)$ are determined.

\medskip

{\bf The paper is organized as follows.} In Section \ref{sec-centralizer} we describe a basis of $\centr (L)$ as a $\coC [L]$-module and the rank of   $\centr (L)$ in terms of the cyclic group of orders of the ODOs in  $\centr (L)$. We extend K. Goodearl's theorem to describe a $\coC [L]$-basis of the maximal submodule $\cM_{\rho}$ of $\centr (L)$ of rank $\rho=(n,M)$, in Theorem \ref{thm-Maxbasis}. Section \ref{sec-AlmostCommuting} is devoted to presenting the basis of almost commuting operators and the systems $\gd_{n,m}$ of the stationary Gelfand-Dickey hierarchy.

In Section \ref{sec:gradedbasis} we define an specialization map $\cE_L$ to evaluate the systems of the GD hierarchy at the coefficients of $L$ obtaining systems $\gd_{n,m}(L)$ of linear equations with coefficients in $\Sigma$. We describe their sets of constant solutions, that form a flag of $\coC$-vector spaces. We use the natural filtration of $\centr (L)$ defined by the order of ODOs to optimize the search algorithm for the so called filtered basis of $\centr (L)$ as $\coC[L]$-module, that corresponds to unique solutions of optimized systems of the GD hierarchy. The main algorithm of this paper is Algorithm \ref{alg:graded_basis}, \texttt{Filtered\_basis} whose input is $L$ and the level $M$ of its centralizer, and output is the filtered basis of $\centr (L)$ if $(n,M)=1$, or $\cM_{\rho}$ in general.

In Section \ref{sec:computable} we restrict to computable differential fields $\coC\langle \eta\rangle$. The sets of constant solutions of the linear system $\gd_{n,m}(L)$ with coefficients in $\coC\langle \eta\rangle$ are computed by building an extended linear system with coefficients in $\coC$. As a consequence we design Algorithm \ref{alg:goodearl_basis},  \texttt{Filtered\_basis\_computable}.

Sections \ref{sec-levelvar} and \ref{sec:families} are devoted to \ref{goal:2}. We include a set of parameters $\Theta$ in the coefficient field $\coC(\Theta)\langle \eta\rangle$ of a differential operator $\bbL$. We fix a level $M$ and force $\bbL$ to have level $M$, obtaining level varieties, as cero sets of level ideals generated by polynomials in $\coC[\Theta]$, in Algorithm \ref{alg:get_ideal},  \texttt{Level\_Variety}. We obtain families $\bbL_M$ of algebro-geometric ODOs of level $M$ in Section \ref{sec:families}, with coefficients in $\coC\langle \eta\rangle$ for $\eta$ equal to $x$, $\cosh (x)$ and the Weierstrass $\wp$ function, with derivation $d/dx$. In Section \ref{sec:implementation}, we include implementation details.

\section{Preliminaries}

Let $\bbN$ be the set of positive integers including zero. Denote by $\bbZ_n$ the cyclic group $\bbZ/n\bbZ$ and by $[m]_n$ the class $m+n\bbZ$. We will write $m_1\equiv_n m_2$ to indicate $m_1-m_2\in n\bbZ$.

\medskip

For concepts in differential algebra, related to differential polynomials, we refer to \cite{Kolchin, Ritt}.  A differential ring $(R, \partial )$ is a ring $R$ (assumed associative, with identity) equipped with a specified
derivation $\partial$. The elements $r\in R$ such that $\partial(r) = 0$ are called constants of $R$, and the
collection of these constants forms a subring of $R$. A differential field is a differential ring that is a field. Given the field of complex numbers $\bbC$, we have differential field extensions $\bbC(x)$ or $\bbC\langle\cosh(x)\rangle$, for the derivation $d/dx$.

Given a differential variable $y$ with respect to $R$, we define $y^{(i)}$ as $\partial^i (y)$, for $y\in\bbN$, $i\geq 1$ and $y^{(0)}$ as $y$. The ring of differential polynomials over $R$ in the differential variable $y$ is the polynomial ring
\[R\{y\} =  R[y^{(i)}\mid i \in \bbN].\]
The notation $y'$, $y''$, $y'''$ will also be used for $y^{(i)}$, $i=1,2,3$.
Observe that~$R\{y\}$ is also a differential ring with the natural extended derivation, that we also call~$\partial$ abusing notation.
We can iterate this construction, adding several differential variables.

\medskip

Let us denote by $R((\partial^{-1}))$ the ring of pseudo-differential operators in $\partial$ with coefficients in the differential ring $R$, defined as in~\cite{Good}, see also \cite{Olver} and \cite{Dikii},
\[
R((\partial^{-1}))=\left\{\sum_{-\infty< i\leq n} a_i\partial^i\mid a_i\in  R , n\in\bbZ\right\},
\]
where $\partial^{-1}$ is the inverse of $\partial$ in $R((\partial^{-1}))$, $\partial^{-1}\partial=\partial\partial^{-1}=1$.

Given $A\in R((\partial^{-1}))$, with $a_n\neq 0$ then $n$ is called the \emph{order} of $A$, denoted by $\ord(A)$ and $a_n$ is called the \emph{leading coefficient} of $A$. By convention $\ord (0)=-\infty$.
The \emph{positive part} of $A$ is the differential operator
\begin{equation}
   A_+= \sum_{i=0}^n a_i\partial^i\in R[\partial] \ \textrm{ if } n\geq 0 \ , \ \textrm{ and } A_+=0 \ \textrm{ otherwise  }.
\end{equation}
As it is usual, we denote by $A_- = A - A_+$. Given $A,B\in R((\partial^{-1}))$ of orders $m$ and $n$ respectively, it is well known that their commutator $[A,B]=AB-BA$ has order $ \leq n+m-1$.

\medskip

For any differential ring $(R,\partial)$, we denote by $R[\partial]$ the non-commutative ring of linear differential operators. Observe that \[R[\partial]=\{A \in R((\partial^{-1}))\ |\ A_- = 0\}.\] 

Any nonzero operator $L\in R [\partial ]$ can be uniquely written in the form
\begin{equation}\label{eq-L-general}
    L= a_0 +a_1 \partial + \cdots +a_n \partial^n , \textrm{with } a_i \in R \ \textrm{and } a_n \not=0 .
\end{equation}
If $a_n=1$ then $L$ is called {\sf monic}. A monic operator with $a_{n-1}=0$ is said to be in {\sf normal form}.
For nonzero $Q_1 , \dots , Q_\ell \in R[\partial ]$, the following inequality holds:
 \begin{equation}\label{eq-ord-max}
        \ord (Q_1 +\cdots +Q_\ell )\leq \max \{ \ord (Q_i ) \ | \ 1\leq i \leq \ell \},
    \end{equation}
and assuming $\ord (Q_i ) \not= \ord (Q_j )$,  whenever $i\not=j$,
    \begin{equation}\label{eq-ord-null}
         Q_1 +\cdots +Q_\ell =0  \Rightarrow Q_i =0,\quad   \forall i.
    \end{equation}

\section{Centralizers of ODOs}\label{sec-centralizer}

Given a monic differential operator $L$ in $R[\partial]$, let us denote by $\centr((L))$ its centralizer in $R((\partial^{-1}))$
\begin{equation}
    \centr((L))=\{A\in R((\partial^{-1}))\mid [L,A]=0\}.
\end{equation}
By a famous theorem of I. Schur in \cite{Sch}, stated for analytic coefficients, if $\ord(L)=n$ then $\centr((L))$ is determined by the unique monic $n$-th root $L^{1/n}$ of $L$.
For coefficients in a differential ring $R$, whose ring of constants is a field $\coC$, this result was proved in \cite[Theorem~3.1]{Good}
\begin{equation}\label{eq-SchurThm}
    \centr((L))=\left\{ \sum_{-\infty<j\leq m} c_j (L^{1/n})^j\mid c_j\in {\bf C}, m\in\bbZ\right\}.
\end{equation}
As a consequence, $\centr((L))$ is a commutative ring. To review the long history of this theorem, see ~\cite[Sections~3 and~4]{Good}.

\medskip

We say that $A \in R[\partial]$ {\sf almost commutes with $L$} if $\ord([L,A]) \leq n-2$, see \cite{Wilson1985} and \cite{JRZHD2025}. We denote by $W(L)$ the set of all differential operators that almost commute with $L$. Assuming that $L$ is monic,  it was proved in \cite{JRZHD2025}, Theorem 3.5, that
\begin{equation}\label{eq-WL}
   W(L)=\centr ((L))_+
\end{equation}
the set of the positive parts of the pseudo-differential operators in $\centr ((L))$ and
\begin{equation}\label{eq:wilson_basis}
    \cB(L)=\{B_j=(L^{j/n})_+\mid  j\in \bbN\}
\end{equation}
is a basis of $W(L)$ as a $\coC$-vector space.

\subsection{In differential operator rings}

In this paper, we will compute centralizers of differential operators whose coefficients belong to a differential field $(\Sigma,\partial)$, whose field of constants $\coC$ has zero characteristic. Given $L\in \Sigma [\partial]$, the centralizer of $L$ in $\Sigma [\partial]$
\[\centr (L)=\{A\in \Sigma[\partial]\mid [L,A]=0\}.\]
By \cite{Good}, Corollary 4.4, $\centr (L)$ is commutative, since
\begin{equation}
\centr (L)=\centr ((L))\cap \Sigma [\partial].
\end{equation}

\medskip

Our algorithm to compute $\centr (L)$ relies on the important fact that $\centr (L)$ is $\coC$-vector subspace of $W(L)$ by $\eqref{eq-WL}$,
\begin{equation}\label{eq-Schur2}
    \centr (L) \subset W(L).
\end{equation}
There is a natural filtration of $W(L)$  determined by the order of an operator, namely $W(L)=\cup_{m\in \bbN} W_m(L)$ where $W_m(L)$ is the vector subspace of almost commuting operators with order less than or equal to m.
The centralizer $\centr (L)$ inherits this filtration
\begin{equation}\label{eq-filtration}
    \centr (L)=\cup_{m\in \bbN} \centr_m (L),\,\,\, \centr_m (L)=W_m(L)\cap \centr (L).
\end{equation}

Observe that $\centr (L)$ contains $\coC [L]$, the ring of all polynomials in $L$ with coefficients in $\coC$.

\begin{definition}\label{def-level}
Let us consider $L\in \Sigma [\partial]\backslash \coC[\partial]$ of order $n$:
\begin{enumerate}
\item  We say that $\centr (L)$ is {\sf trivial} if $\centr (L)=\coC[L]$ and  {\sf trivial at level $m$} if $\centr_m (L)\subset \coC[L]$.

\item We define the {\sf level} of $L$ to be the smallest $M$ such that $\centr(L)$ is not trivial at level $M$, that is $\centr_M (L)$ is not included in $\coC [L]$.
\end{enumerate}
\end{definition}

\medskip

It is important to note that  $\centr (L)$ is a domain, which immediately follows from 1 in the next lemma, proved in \cite{Good}, Lemma 1.1.
\begin{lemma}\label{lem-PQ}
    Given nonzero operators $P,Q\in\centr(L)$, then:
    \begin{enumerate}
        \item $P\cdot Q\neq 0$ and $\ord(PQ)=\ord(P)+\ord(Q)$.

        \item If $\ord(P)=\ord(Q)=m$ then there exists a nonzero $\alpha\in\Sigma$ such that $Q-\alpha P$ has order strictly smaller than $m$.
    \end{enumerate}
\end{lemma}

We state next K. Goodearl's \cite[Theorem~1.2]{Good} in the context that will be used in this paper, for ODOs in $\Sigma [\partial]$. This result describes a basis of $\centr (L)$ as a $\coC [L]$-module. The proof was reviewed in \cite[Appendix~A]{RZ2024} and in this occasion we decompose the proof to extract further conclusions.

\begin{lemma}\label{lem-CL}
    Let us consider $L\in \Sigma [\partial]$ of order $n$. Then
    \begin{equation}
        \coC[L]=\{Q\in \centr (L)\mid \ord (Q)\in [0]_n\}.
    \end{equation}
\end{lemma}
\begin{proof}
Let $\upsilon$ be the leading coefficient of $L$.
Given $Q\in\centr (L)$, with order $m \in n\bbZ$, let us use induction on $m$.
If $m=0$, $Q=\sigma\in \Sigma$, the commutator $LQ-QL$ is a  differential operator of order $n-1$, whose leading coefficient $n \upsilon \partial(\sigma)$ must be zero in $\Sigma$.
Thus $Q=\sigma\in \coC$.
If $m>0$, by Lemma \ref{lem-PQ} there exists a nonzero $\alpha\in \coC$ such that $Q-\alpha L^{m/n}$ has order divisible by $n$ and smaller than $m$. By induction hypothesis $Q-\alpha L^{m/n}$ belongs to $\coC [L]$,  and so does $Q$.
\end{proof}

\begin{lemma}\label{lem-li}
    Any finite family $\{G_k\}$ of differential operators in $\centr(L)\backslash \coC [L]$, with $\ord (G_i)\not\equiv_n \ord (G_j)$, for $i\neq j$, is $\coC [L]$-linearly independent.
\end{lemma}
\begin{proof}
For a finite sum $\sum_k p_k(L) G_k $, $p_k(L)\in \coC[L]$ to be zero, it is necessary that all summands are zero, due to \eqref{eq-ord-null} applied to $p_k(L) G_k$. But then $p_k(L) =0$, since $\centr (L)$ is a domain. Consequently $\{G_k\}$ is a finite family of $\coC[L]$-linearly independent differential operators
\end{proof}

\begin{definition}
Consider $L \in \Sigma [\partial]$ of order $n$.
    Given $G\in \centr (L)$ we say that $G$ is {\sf order minimal in $\centr (L)$} if $\ord(G)$ is minimal with respect to the orders of all operators $Q$ in $\centr (L)$ with $\ord(Q)\equiv_n \ord(G)$.
\end{definition}

\begin{lemma}\label{lem-O}
     Given a differential operator $L \in \Sigma [\partial]$ of order $n$, then the set
     \begin{equation}\label{eq-O}
       \cO=\{[\ord(Q)]_n\mid Q\in\centr(L), Q\neq 0\}
     \end{equation}
is a cyclic subgroup of $\bbZ_n$, whose order $d=|\cO|$ divides $n$. We call $\cO$ {\sf the group of orders of $\centr (L)$}.
\end{lemma}
\begin{proof}
The set of orders $\{\ord(Q)\mid Q\in\centr(L), Q\neq 0\}$ is closed under addition by Lemma \ref{lem-PQ}, 1. Therefore $\cO$ is an additively closed subgroup of the cyclic group $\bbZ_n$. By any introduction to group theory, $\cO$ is a cyclic subgroup of $\bbZ_n$ and its order $d$ is a divisor of $n$.
\end{proof}

Observe that having a non-trivial centralizer $\centr (L)\neq \coC [L]$
is equivalent to having a non-trivial group of orders $\cO\neq \{[0]_n\}$ .

\begin{theorem}[Goodearl's Theorem]\label{thm-cenbasis}
Let us consider $L \in \Sigma [\partial]$ of order $n$ and the group $\cO$ of orders of $\centr (L)$, as in \eqref{eq-O}.
Then a basis of $\centr (L)$ as a $\coC[L]$-module consists of $d=|\cO|$ operators
   \begin{equation}
     \cG(L):=\{ G_0=1, G_1,\ldots,G_{d-1}\}
   \end{equation}
satisfying the following conditions:
\begin{enumerate}
 \item Each $G_k$ is order minimal in $ \centr (L)$.

\item $\cO=\{[\ord(G_k)]_n\mid k=0,1,\ldots ,d-1\}$.
\end{enumerate}
\end{theorem}
\begin{proof}
By lemmas \ref{lem-O} and \ref{lem-li}, a set $\cG(L)$ verifying these conditions is a linearly independent set. It only remains to prove that it is a set of generators
    \begin{equation}
        \centr (L)=\bigoplus_{k=0}^{d-1} \coC[L] G_k=\coC[L,G_1,\ldots , G_{d-1}].
    \end{equation}
Instead of repeating here the proof given in \cite[Theorem~1.2]{Good}, we will prove an extension of this result in Theorem \ref{thm-Maxbasis}.
\end{proof}

\medskip

The rank of a set of differential operators is defined to be the greatest common divisor of all orders, see \cite{PRZ2019}. Thus
the {\sf rank of the centralizer} is
\begin{equation}
    \rank (\centr (L))=\gcd \{\ord(Q)\mid Q\in \centr (L)\}.
\end{equation}

\begin{cor}
    Let us consider $L \in \Sigma [\partial]$ of order $n$ and let $d=|\cO|$. Then
    \begin{equation}
        \rank (\centr (L))=\frac{n}{d}.
    \end{equation}
\end{cor}
\begin{proof}
Since the order of $\cO$ is $d$ then $\cO$ is the cyclic subgroup of $\bbZ_n$ generated by $[n/d]_n$, which proves this result.
\end{proof}

\medskip

For $n\geq 2$, the definition of algebro-geometric operator  given in \cite{We} is equivalent to the existence of $A\in \centr(L)\backslash \coC [L]$ of order coprime with $n$. We can give the following equivalent definition.

\begin{definition}\label{def-ag}
    A differential operator $L\in\Sigma [\partial]\backslash \coC[\partial]$ is {\sf algebro-geometric} if $\rank(\centr (L))=1$.
\end{definition}

By Theorem \ref{thm-cenbasis} and the observation that
\[\rank(\centr (L))=1\Leftrightarrow n=|\cO|,\]
the following corollary follows.

\begin{cor}\label{thm-cenbasis_prime}
    Let $L \in \Sigma [\partial]\backslash \coC[\partial]$ be algebro-geometric.  Then a basis of $\centr (L)$ as a $\coC[L]$-module is formed by $n$ operators $A_i \in \Sigma[\partial]$, order minimal in $\centr (L)$ and  such that $\ord(A_i) \equiv_n i$, $i=0,\ldots ,n-1$.
\end{cor}

\begin{rem}\label{rem-Zheglov}
If $n=\ord(L)$ is prime and $\centr (L)$ is non-trivial then $L$ is algebro-geometric, it has a centralizer of rank $1$.

There are also algebro-geometric operators whose order is not prime.
Consider the order $4$ operator $L$ given in ~\cite[Example 22]{PZ}, included in Example \ref{exm:Zheglov} of this paper.

The centralizer of $L$  was computed with the algorithm developed in this paper, Algorithm \ref{alg:graded_basis} and
is $\centr (L)=\coC [L_4,A_1,A_2,A_3]$
with
$\ord(A_2)=6$, $\ord(A_3)=7$ and $\ord(A_1)=9$, see Example \ref{exm:Zheglov}. So even if $\rank(L,A_2)=2$ the centralizer $\centr (L)$ is of rank $1$. Computing a basis of $\centr (L)$ allowed us to determine that $L$ is algebro-geometric.
\end{rem}

\subsection{Maximal submodule of the cetralizer}

The algorithms developed in this paper combine \eqref{eq-Schur2} with Theorem \ref{thm-cenbasis} to compute centralizers.
The inputs will be a  differential operator $L$ of order $n$ in normal form and its level $M$.
If $n$ and $M$ are coprime then the rank of the centralizer is $1$ and our algorithm computes a basis of the whole centralizer. In the case
$$\rho=(n,M)> 1$$ our algorithm is guaranteed to compute a subalgebra of rank $\rho$ of the centralizer, in fact the {\sf maximal commutative subalgebra of $\centr (L)$ of rank $\rho$}
\begin{equation}
    \cM_{\rho}=\{Q\in \centr (L)\mid \rho | \ord(Q)\},
\end{equation}
with group of orders
\begin{equation}\label{eq-oM}
    \cO_{\rho}=\{[\ord(Q)]_n \mid Q\in\cM (L), Q\neq 0\}.
\end{equation}
We have the following chain of $\coC [L]$-modules
\[\coC[L]\subseteq \cM_{\rho} \subseteq \centr (L),\]
and an upper bound of the rank of the centralizer
\[\rank (\centr (L))\leq \rank (\cM_{\rho})=\rho.\]
Note that $\centr (L)=\cM_{\rho}$ if and only if $(n,M)=1$ is the rank of the centralizer.

\begin{lemma}\label{lem-Grho}
Consider $L \in \Sigma [\partial]$ of order $n$ with $\centr (L)\neq \coC[L]$.
Given $G\in \centr (L)\backslash \coC [L]$, let $\rho=(n,\ord(G))$ and $d=n/\rho$. Then there exists
\begin{equation}
    \cG_{\rho}=\{G_k \in \centr (L )\mid k=0,\ldots ,d-1\}
\end{equation}
such that each $G_k$ is order minimal in $\centr (L)$, with $\ord(G_k)\equiv_n k \rho$ and
\[\ord(G_k)\leq (d-1) \ord(G).\]
\end{lemma}
\begin{proof}
Let $m=\ord(G)$. We know that $G^k$ belong to $\centr (L)\backslash \coC [L]$ and $\ord(G^k)\equiv_n  k m$, $k=2,\ldots ,d-1$. Thus there exists $G_k$ order minimal in $\centr (L)$ for each $k$, with
\[\ord(G_k)\leq k m \leq (d-1)m.\]
\end{proof}

The previous lemma proves existence of a $\coC[L]$-linearly independent set of ODOs in the centralizer, that will be proved to be a basis of $\cM_{\rho}$ in the next theorem.

\begin{theorem}\label{thm-Maxbasis}
Consider $L \in \Sigma [\partial]$ of order $n$ and level $M$, with $\rho=(n,M)$.
Then a basis of $\cM_{\rho}$ as a $\coC[L]$-module consists of $d=n/\rho$ operators,
   \begin{equation}
     \cG_{\rho}:=\{ G_0=1, G_1,\ldots,G_{d-1}\}
   \end{equation}
satisfying the following conditions:
\begin{enumerate}
\item Each $G_k$ is order minimal in $\centr (L)$.

\item The set of orders $\cO_{\rho}$, as in \eqref{eq-oM}, is an additive group isomorphic to $\bbZ_d$, $$\cO_{\rho}=\{[\ord(G_k)]_n\mid k=0,1,\ldots ,d-1\}.$$
\item $\ord (G_k)< (d-1) M$.
\end{enumerate}
\end{theorem}
\begin{proof}
By Lemma \ref{lem-Grho} there exists a set $\cG_{\rho}$ verifying the given conditions.
By Lemma \ref{lem-li} any set $\cG_{\rho}$ verifying these conditions is a $\coC [L]$-linearly independent set and $\cM_{\rho}$ contains the free $\coC [L]$-submodule
\begin{equation*}
    W:= \bigoplus_{k=0}^{d-1} \coC [L]G_k.
\end{equation*}
Let $Q$ be a differential operator in $ \cM_{\rho}$ of order $m$. By Lemma \ref{lem-CL}, we can restrict to $m\notin n\bbZ$. We will proceed by induction on $m$ to show that $Q$ is in $W$.

The operator of smallest order divisible by $\rho$ in $\centr (L)\backslash \coC[L]$ is $G_1$. Let us now consider an operator $Q$ of order $m>M$ in $\cM_{\rho}$. There exists $G_i$ such that $\ord(Q)\equiv_n \ord (G_i)$ and we can define $T_Q :=L^q G_{i}$ of order $m$ for some $q$. Consequently, by Lemma \ref{lem-PQ}, 2 there exists a nonzero constant $\alpha\in\coC$ such that $Q-\alpha  T_Q$ is an operator of order $<m$ in $\cM_{\rho}$. According to the induction hypothesis, $Q-\alpha T_Q$ belongs to $W$ and also does $Q$. Consequently $\cM_{\rho}=W$.
\end{proof}

\section{Almost commuting bases and Gelfand-Dickey hierarchies}\label{sec-AlmostCommuting}

Let us consider the field of constants $\coC$ with respect to $\partial$, and assume that $\coC$ has zero characteristic, thus $\coC$ contains (a field isomorphic to) the field of rational numbers $\bbQ$.
Let $L_n$ be a formal differential operator in normal form
\begin{equation}\label{def-L-formal-U}
 L_n = \partial^n + u_2 \partial^{n-2} + \ldots + u_{n-1}\partial + u_n,
\end{equation}
which belongs to $\bbQ\{U\}[\partial ]\subset \coC\{U\}[\partial ]$, with a set $U = \{u_2,\ldots, u_n\}$ of differential variables with respect to the derivation $\partial$. In the ring of differential polynomials $\coC\{U\}$,  we consider the extended derivation
\begin{equation}\label{def-w-cR}
    \partial^k (u_\ell ) =u_\ell^{(k)} , \quad k\in \bbN.
\end{equation}

See \cite[Chapter 1, \S 7]{Kolchin}, for a definition of differential grading of $\coC\{U\}$.
We define a weight function $w$ on  $\coC\{U\}[\partial]$ by setting
\begin{equation}\label{eq-weight}
w(u_\ell^{(k)})=\ell+k , \ k\in \bbN \mbox{ and }w(\partial) = 1.
\end{equation}
From the homogeneity of the commutation rule $\partial p= p\partial +\partial (p)$, for every $p\in \coC\{U\}$ it follows that the product of two operators that are homogeneous of weights $r$ and $s$ is homogeneous of weight $r + s$. As in \cite[Proposition~4.2]{SW}, we call an operator $P\in \coC\{U\}[\partial]$ {\sf homogeneous of weight $r$ with respect to $w$} if the coefficient of $\partial^i$ in $P$ has weight $r-i$.

Let $Q=L^{1/n}$ be the unique monic $n$-th root of $L$ in the ring of pseudo-differential operators $\coC\{U\}((\partial^{-1}))$ and let
$P_m=(Q^m)_+$ be the positive part of $Q^m$.
By \cite[Theorem~3.6]{JRZHD2025}, which is based on \cite[Proposition~2.4]{Wilson1985},
\begin{equation}
    \cH(L)=\{P_m=(Q^m)_+ \mid\ m \in \bbN\}
\end{equation}
is the unique basis of  $W(L)$ for which each $P_m\in \cH(L)$ is homogeneous of weight $m$. To be more precise
\begin{equation}
    P_m=\partial^m+ q_2(U) \partial^{m-2}+\cdots +q_{m-1}(U)\partial +q_m(U), \mbox{ with } w(q_j)=j.
\end{equation}
We call $\cH(L)$ the {\sf homogeneous basis} of $W(L)$.

\medskip

In addition, let us denote by $\overline{H}_m=(H_{0,m},\ldots, H_{n-2,m})$ the vector of differential polynomials such that
    \[[P_m,L_n] = H_{0,m}(U) + H_{1,m}(U)\partial + \ldots + H_{n-2,m}(U)\partial^{n-2}.\]
Then each $H_{k,m}(U) \in \bbQ\{U\}$ is a homogeneous differential polynomial of weight $m+n-k$, see for instance \cite[Section~5]{JRZHD2025}.

\begin{rem}
    The computation of the operators $P_m$ and the vectors of differential polynomials $\overline{H}_m$ has been implemented in the SageMath package \texttt{dalgebra} based on the algorithms designed in \cite{JRZHD2025}.
\end{rem}

Let $L_t$ denote the differential operator obtained from $L_n$ by computing the partial derivative of its coefficients with respect to a variable $t$. Given a monic almost commuting operator $A_m$ of order $m$, that is $A_m\in W_m(L)\backslash W_{m-1}(L)$,  the Lax equation for $L_n$ and $A_m$ is
 \begin{equation}\label{eq-Lax}
     L_t = [A_m , L_n] \ , \ \textrm{ with } t=t_m \ .
 \end{equation}
If the differential indeterminates in $U$ are assumed to be constant for the derivation with respect to $t$ we obtain the {\sf stationary Lax equation}
\begin{equation}\label{eq-Lax-est0}
     0=[A_m , L_n] \ .
 \end{equation}

The equation \eqref{eq-Lax-est0} can be written in terms of the elements of the homogeneous
$\coC$-basis $\cH(L)$  of $W(L)$. In fact,
\begin{equation}\label{eq-Lax-A-basic}
    A_m =P_m +\xi_{m-1}P_{m-1}+\dots + \xi_{0}P_{0} ,\,\,\, \xi=(\xi_0,\ldots ,\xi_{m-1}) \in \coC^m.
\end{equation}
Observe that if $m\in n\bbZ$ then $P_m=L^{m/n}$ is an integer power of $L$, thus $[P_m,L_n]=0$. Then, we can always write
\begin{equation}\label{eq-Lax-est}
     0=[A_m , L_n]=[A^*_m , L_n] \ ,
 \end{equation}
where we define
\begin{equation}\label{eq-Lax-A}
    A^*_m =P_m +\sum_{j\in J_{m-1}} \xi_{j} P_j,\mbox{ with } J_{n,m-1}=\bbZ_{m-1}\backslash [0]_n.
\end{equation}

The Lax equations \eqref{eq-Lax-est} provide a system of   differential equations defined by non linear differential polynomials in the  differential variables $U$, obtained from the coefficient of $\partial^{k}$ in $[A^*_m,L_n]$
\begin{equation}\label{eq-GD}
    0 = H_{k,m}(U) +\sum_{j\in J_{n,m-1}} c_{j}  H_{k,j}(U) \ , \textrm{for } \  k = 0, \dots , n-2,\,\,\, m\geq 2.
\end{equation}
Let us write these systems in row vector form.

For a fixed $n$, the family \eqref{eq-GD} is called the {\sf  stationary Gelfand-Dickey (GD) hierarchy of $L_n$}, see \cite{Dikii, DrinfeldSokolov1985, Wilson1985} and references therein. For a fixed $m\notin n\bbZ$,  {\sf the $(n,m)$-system of the Gelfand-Dickey hierarchy} is denoted by
\begin{equation}\label{eq-Hnm}
   \gd_{n,m}: \overline{H}_m+\sum_{j\in J_{n,m-1}} c_{j}  \overline{H}_j= \overline{0}
\end{equation}
where $\overline{H}_j(U)$ is the row vector $(H_{0,j}(U),\ldots ,H_{n-2,j}(U))$ and
\[\cC_m=\{c_j,\ j\in J_{n,m-1}\}  \]
is a set of algebraic variables.
We will also say that $ \gd_{n,m}$  is the {\sf GD hierarchy of $L_n$ at level $m$}.

Some of these hierarchies have been given specific names. For
$n=2$, the GD hierarchy of $L_2=\partial^2+u_2$ is the Korteweg-de Vries (KdV) hierarchy
\begin{equation}\label{eq-KdV}
    \gd_{2,m}:\, H_{0,m}(U) +\sum_{j\in J_{2,m-1}} c_{j}  H_{0,j}(U)=0\ ,\,\,\, m\geq 2, m\notin 2\bbZ.
\end{equation}
For $n=3$, the Boussinesq hierarchy is obtained.
The GD hierarchy of $L_3= \partial^3 + u_2\partial + u_3$ consists of systems of two nonlinear equations. For each $m\geq 2$, $m\notin 3\bbZ$ and $U=\{u_2,u_3\}$, the system in row vector form is
\begin{equation}
    \gd_{3,m}:\, (H_{0,m}(U), H_{1,m}(U)) +\sum_{j\in J_{3,m-1}} c_{j}  (H_{0,j}(U),H_{1,j}(U))=(0,0).
\end{equation}

\begin{example}\label{exm:hierarchy:3:4}
    Let us write the $(3,4)$-sytem of the Boussinesq hierarchy. Following  ~\cite{JRZHD2025}, from the formal differential operator $L_3 = \partial^3 + u_2 \partial + u_3$ we compute the homogeneous basis of almost-commuting operators $\{P_1,P_2,P_3, P_4\}$ of $W_4(L_3)$, where $P_3=L_3$,
    \begin{align*}
    &P_1 = \partial,\,\,\, P_2 = \partial^2 + \frac{2}{3}u_2 \mbox{ and }\\
    &P_4 = \partial^4 + \left(\frac{4}{3}{u_2}\right)\partial^2 + \left(\frac{2}{3}{u_2}' + \frac{4}{3}{u_3}\right)\partial + \left(\frac{2}{9}{u_2}^{2} + \frac{2}{9}{u_2}'' + \frac{2}{3}{u_3}'\right).
\end{align*}
The $(3,4)$-system of the stationary GD hierarchy, $\gd_{3,4}$ is the linear system in $c_1$ and $c_2$ defined by
    \[\gd_{3,4}:\left\{\begin{array}{rl}
        c_1 H_{0,1} + c_2 H_{0,2} + H_{0,4} & = 0,\\
        c_1 H_{1,1} + c_2 H_{1,2} + H_{1,4} & = 0.\\
    \end{array}\right.\]
    where $H_{k,j}$ are the differential polynomials in $\coC\{u_2,u_3\}$ that will be used to construct illustrating examples throughout this paper,
   \begin{align*}
     &\overline{H}_1 = (-u_3', -u_2'),\,\,\, \overline{H}_2=((2/3){u_2}{u_2'} + (2/3){u_2}''' - {u_3}'',{u_2}'' - 2 {u_3}' ) \mbox{ and }\\
     &\overline{H}_4 = (H_{0,4}, H_{1,4})\mbox{ where }\\
      &      H_{0,4} ={} (2/3){u_2}{u_2}''' - (2/3){u_2}{u_3}'' + (4/9){u_2}^{2}{u_2}' + (4/3){u_2}'{u_2}'' - (2/3){u_2}'{u_3}' \\ & + (2/9){u_2}^{(5)} - (4/3){u_3}{u_3}' - (1/3){u_3}^{(4)},\\
      &H_{1,4} ={} (2/3){u_2}{u_2}'' -(4/3){u_2}{u_3}' - (4/3){u_2}'{u_3} + (2/3){u'_2}^{2} + (1/3){u_2}_{4} -(2/3){u_3}'''.
    \end{align*}
\end{example}

\section{The filtered basis of the centralizer}\label{sec:gradedbasis}

Every $\coC [L]$-basis of the centralizer verifies the conditions of Theorem \ref{thm-cenbasis}. In this section, we will optimize the procedure to search for a $\coC [L]$-basis of the centralizer and describe a basis uniquely determined by the natural filtration of $\centr (L)$ defined by the order of differential operators \eqref{eq-filtration}.

\medskip

Let us consider a differential field $(\Sigma,\partial)$, whose field of constants $\coC$ has zero characteristic.
Let $L$ be a  differential operator in normal form
\begin{equation}\label{def-L}
 L = \partial^n + \upsilon_2 \partial^{n-2} + \ldots + \upsilon_{n-1}\partial + \upsilon_n, \mbox{ with }\upsilon_j\in \Sigma.
\end{equation}
Given the coefficient vector $\upsilon_L=(\upsilon_2,\ldots ,\upsilon_n)$ of $L$, let us consider the specialization map
\begin{equation}\label{eq:cE}
    \cE_{L}: \coC\{U\}\rightarrow \coC\{\upsilon_L\}\subset \Sigma, \mbox{ defined by } \cE_{L}(u_i)=\upsilon_i, i=2,\ldots ,n.
\end{equation}
It is a homomorphism of differential rings and it extends naturally to an specialization of differential operators $\coC\{U\}[\partial]\rightarrow \coC\{\upsilon_L\}[\partial]$. Observe that $L=\cE_L(L_n)$ is the specialization of the formal differential operator in \eqref{def-L-formal-U}.

For each $j\in J_{n,m}=\bbZ_m\backslash [0]_n$ let us define the specialization of the differential polynomials that define the $(n,m)$-system of the GD hierarchy
\begin{equation}\label{eq-sigmakj}
    \sigma_{k,j}:=\cE_L(H_{k,j}),\ k=0,\ldots , n-2,\, j\in J_{n,m}.
\end{equation}
By specializing the differential operators of the homogeneous basis $\cH(L_n)$ we obtain a basis $\cB(L):=\{B_j\mid j\in\bbN\}$ of $W(L)$ where
\begin{align}\label{eq-ACbasis}
B_j:=\cE_L(P_j),\ j\in \bbN.
\end{align}

\begin{definition}
Given $m\geq 2$ and not in $[0]_n$,
the {\sf $(n,m)$-system of the GD hierarchy for $L$} is the result of specializing the differential polynomials in \eqref{eq-Hnm}
\begin{equation}\label{equ:gd}
    \gd_{n,m} (L): \overline{\sigma}_m+\sum_{j\in J_{n,m-1}} c_{j} \overline{\sigma}_j=\overline{0},\mbox{ with } \overline{\sigma}_j=(\sigma_{0,j},\ldots ,\sigma_{n-2,j}),
\end{equation}
a system of $n-1$ linear equations in the set of indeterminates $\cC_m=\{c_{j}\mid j\in J_{n,m-1}\}$, with coefficients in $\Sigma$.
We will also define the {\sf homogeneous $(n,m)$-system of the GD-hierarchy for $L$} by
\begin{equation}\label{equ:hgd}
    \hgd_{n,m} (L): \sum_{j\in J_{n,m}} c_{j} \overline{\sigma}_j=\overline{0}.
\end{equation}
\end{definition}

\begin{example}[$n=3$, $m=4$]\label{exm:hierarchy:3:4:eval}
We define the corresponding specialization map $\cE_L: \coC\{u_2,u_3\}\rightarrow \Sigma$.
By specializing the formal differential operators in Example \ref{exm:hierarchy:3:4} we obtain a basis of $W_4(L)$, formed by the differential operators $B_j=\cE_L(P_j)$, $j=1,\ldots ,4$. By specializing the differential polynomials in Example \ref{exm:hierarchy:3:4}, we obtain $\sigma_{k,j}=\cE_L(H_{k,j})$. Then the homogeneous $(3,4)$-system of the GD-hierarchy for $L$ is a linear system in  $c_1,c_2$ and $c_4$
\[
\hgd_{3,4}(\bbL):
\begin{cases}
c_1\sigma_{0,1}+c_2\sigma_{0,2}+c_4\sigma_{0,4}=0,\\
c_1\sigma_{1,1}+c_2\sigma_{1,2}+c_4\sigma_{1,4}=0.
\end{cases}
\]
\begin{enumerate}
\item \textbf{(Rational)} Consider the differential field $\Sigma = \bbC(x)$ with the standard derivation $\partial = \frac{d}{dx}$ and the differential operator
    $$L = \partial^3 - \frac{6}{x^2}\partial + \frac{12}{x^3}.$$
    The specialization map $\cE_L$ is defined by $\cE_L(u_2) = -6/x^2$ and $\cE_L(u_3) = 12/x^3$.
In this case  $\cE_L(\overline{H}_4) = (0,0)$ and
    \[\hgd_{3,4}(L) : \left\{\begin{array}{rl}
        \frac{36}{x^4} c_1  -\frac{96}{x^5} c_2 &= 0,\\ & \\
        -\frac{12}{x^3} c_1 + \frac{36}{x^4} c_2 &= 0.
    \end{array}\right.\]

    \item \textbf{(Hyperbolic)} Consider the differential field $\Sigma=\bbC \langle\cosh(x)\rangle$ with derivation $\partial=d/dx$ and
    the differential operator
    $$L = \partial^3 +\frac{6}{\cosh(x)^2}\partial.$$
    The specialization map $\cE_L$ is defined by $\cE_L(u_2)  = \frac{6}{\cosh(x)^2}$, $\cE_L(u_3) = 0$. Denoting $\eta=\cosh(x)$
\[\hspace{-\dimexpr\oddsidemargin-0.1in}%
    \hgd_{3,4}(L) : \left\{\begin{array}{rcrcrcl}
       && \frac{\left(48 - 32 \eta^2 \right) \eta'}{ \eta^5} c_2 & + &\frac{\left(-128 \eta^2 + 192\right) \eta'}{3 \eta^5}c_4 &=& 0,\\
        \frac{12 \eta'}{ \eta^3} c_1 & + &\frac{24 \eta^2 - 36}{ \eta^4} c_2 & + &\frac{32  \eta^2 - 48}{ \eta^4}c_4&=& 0.%
    \end{array}\right.\]
\end{enumerate}
\end{example}

\subsection{Flag of constants of the GD hierarchies}\label{sec:flag}

We describe next the $\coC$-vector spaces of constant solutions of the linear systems $\gd_{n,m}(L)$, whose coefficients belong to $\Sigma$ and whose set of indeterminates is
\[\cC_m=\{c_j\mid j\in J_{n,m-1}\}, \mbox{ with } J_{n,m-1}=\bbZ_{m-1}\backslash [0]_n.\]
The results in this section are a generalization of \cite{MRZ1}, Section 4.1, where the flag of constants is described only for the KdV hierarchy.

\medskip

Observe that the linear system $\gd_{n,m}(L)$ in the unknowns $\cC_m$
determines an affine subspace of $\Sigma^{|J_{n,m-1}|}$. Let $F_m$ be its intersection with $\coC^{|J_{n,m-1}|}$
\begin{equation}\label{eq-Fm}
    F_m:=\{\xi\in \coC^{|J_{n,m-1}|}\mid \xi \mbox{ is a solution of } \gd_{n,m}(L)\}.
\end{equation}
In addition, consider $\coC$-vector spaces
\begin{equation}\label{eq-Vm}
    V_m:=\{\xi\in \coC^{|J_{n,m}|}\mid \xi \mbox{ is a solution of } \hgd_{n,m}(L)\}.
\end{equation}

Recall the filtration of $\centr (L)= \cup_m \centr_m (L)$ in \eqref{eq-filtration}.
For every $m\in\bbN$, let us denote by $\coC_m[L]$ the $\coC$-vector space of polynomials in $L$ of order less than or equal to $m$,
which is a $\coC$-vector subspace of $\centr_m (L)$.

\begin{proposition}\label{prop-Vm}
Let $L$ be a differential operator in $\Sigma[\partial]$, of order $n$ and in normal form.
For every $m\in\bbN$, there is an isomorphism of $\coC$-vector spaces
\begin{equation}\label{eq-iso}
    V_m\simeq \frac{\centr_m(L)}{\coC_m [L]}.
\end{equation}
\end{proposition}
\begin{proof}
Observe that, if $m\in n\bbZ$ then $V_m=V_{m-1}$, with $m-1\notin n\bbZ$. Thus in this case $\centr_m(L)/\coC_m [L]$ is equal to $\centr_{m-1}(L)/\coC_{m-1} [L]$.

Given $m\notin n\bbZ$, there is a natural
isomorphism of $\coC$-vector spaces between $\coC^{m+1}$ and $W_m(L)$. Taking into consideration \eqref{eq-Lax-A}, we obtain the
isomorphism
\begin{equation}\label{equ:isomorphism}  \coC^{|J_{n,m}|}\simeq \frac{W_m(L)}{\coC_m[L]}, \mbox{ defined by } \xi\mapsto \sum_{j\in J_{n,m}} \xi_j B_j+\coC_m[L],
\end{equation}
that restricted to $V_m$ becomes the desired isomorphism, since
\[\sum_{j\in J_{n,m}} \xi_j B_j\in\centr_m(L)=W_m(L)\cap \centr(L).\]
\end{proof}

If $V_m$ is zero then $\hgd_{m,n}(L)$ has only the zero solution, and by Proposition \ref{prop-Vm}, then $\centr_m(L)=\coC_m[L]$. The following conclusion follows inmediately.

\begin{cor}
Let $L$ be a differential operator in $\Sigma[\partial]$, of order $n$ and in normal form.
    If for every positive integer $m\notin [0]_n$ the system $\hgd_{n,m}(L)$ has only the zero solution then  $\centr (L)=\coC [L]$.
\end{cor}

We can now rephrase the definition of the level of an operator $L$, Definition \ref{def-level}, as the  smallest positive integer $M\notin [0]_n$ such that $\gd_{n,M}(L)$ has a solution in $\coC^{|J_{n,M-1}|}$. That is the smallest positive integer $M$, such that $F_M\neq \emptyset$.

\begin{lemma}\label{lem-flag}
Let $L$ be a differential operator in $\Sigma[\partial]$, of order $n$ and in normal form. Let $M$ be the level of $L$.
\begin{enumerate}
    \item
 Identifying $V_m$ with its natural embedding in $V_{m+1}$,  there is an infinite flag of $\coC$-vector spaces
    \begin{equation}
        V_{M} \subseteq \cdots \subset V_{m} \subseteq V_{m+1} \subseteq \cdots,
    \end{equation}
    we call it {\sf the flag of constants of $L$}.

    \item $V_m=V_{m-1}\oplus T_m$ with $\dim T_m\leq 1$. If $\dim T_m=1$, let  $T_m=\langle (\bc_m,1)\rangle$.

    \item $F_m=\emptyset$ or $F_m=\{\bc_m\}+V_{m-1}$ is the affine space with associated vector space $V_{m-1}$. In particular $F_{M}=\{\bc_M\}$ and $V_{M}=\langle (\bc_M,1)\rangle$.
\end{enumerate}
\end{lemma}
\begin{proof}
\begin{enumerate}
    \item Recall that if $m+1\in n\bbZ$ then $V_m=V_{m+1}$, otherwise there is a natural embedding from $V_m$ to $V_{m+1}$ by $\xi\mapsto (\xi,0)$.

    \item If $m\notin n\bbZ$ then there exists $(\bc_m,1)\in V_m$, with $\bc_m=(\bc_{m,j_1},\ldots ,\bc_{m,j_s})\in \coC^{J_{n,m-1}}$, with $J_{n,m-1}=\{j_1,\ldots ,j_s\}$. Given a nonzero $\xi\in T_m$, there exists $\alpha\in\coC$ such that $\alpha (\bc_m, 1)- \xi\in V_{m-1}\times \{0\}\cap T_m$. Thus $\xi=\alpha (\bc_m, 1)$.

    \item It is an immediate consequence of 2.
\end{enumerate}
\end{proof}

\begin{lemma}\label{lem-quotient}
    Given $m\in\bbN$, if $m\notin [0]_n$ then there is an isomorphism of $\coC$-vector spaces
    \begin{equation}
        \frac{V_m}{V_{m-1}}\simeq \frac{\centr_m(L)}{\centr_{m-1}(L)}.
    \end{equation}
If $m\in [0]_n$ then $V_m=V_{m-1}$.
\end{lemma}
\begin{proof}

Observe that if $m\notin n\bbZ$ then $\coC_m[L]=\coC_{m-1}[L]$. The result follows by the third isomorphism theorem
    \begin{equation}
        \frac{V_m}{V_{m-1}}\simeq \frac{\centr_m(L)/\coC_{m-1}[L]}{\centr_{m-1}(L)/\coC_{m-1}[L]}.
    \end{equation}
\end{proof}

\subsection{Optimized GD hierarchies}

The next results are devoted to describe the basis that is uniquely determined by the filtration of the centralizer. Such basis is obtained from the unique solutions of an optimization of the systems $\gd_{n,m}(L)$.

\medskip

Let $\cG(L)=\{G_0=1,G_1,\ldots ,G_{d-1}\}$ be a basis of $\centr (L)$ and assume that
\begin{equation}\label{eq-ordseq}
M_0=0<M_1=\ord(G_1)<M_2=\ord(G_2)<\cdots <M_{t-1}=\ord(G_{d-1}).
\end{equation}
The next lemma is a consequence of \eqref{eq-ordseq}.

\begin{lemma}\label{lem-directsum}
Given $m\in\bbN$, the following identity holds:
    \[\centr_{m}(L) = \bigoplus_{G \in \cG(L)} \coC_{m-\ord(G)}[L] G,\]
assuming that $\coC_{m-\ord(G)}[L]$ is zero if $m-\ord(G)<0$.
In particular, for every $i\in \{1,\ldots ,d-1\}$ we can write
\[\centr_{M_i} (L)= \bigoplus_{j=0}^i \coC_{M_i-M_j} [L] G_j.\]
\end{lemma}


\begin{rem}
The following piece of the flag of constants is of special importance
\begin{equation}
        V_{M_1} \subset \cdots \subset V_{M_1+\ell_1 n} \subset \cdots \subset V_{M_2} \subset \cdots \subset V_{M_{d-1}}.
    \end{equation}
Observe that $V_{M_i-1}$ can be identified with $V_{M_{\ell}+h_{\ell} n}$, where $$M_{\ell}+h_{\ell} n=\max\{M_k+h n\mid 0\leq hn<M_i-M_k, k=0,\ldots ,i-1
\}.$$
\end{rem}

The next theorem will allow us to define an iterative algorithm to search for the basis of $\centr (L)$. At each step, we compute the constant solutions of the  {\sf optimized GD $(n,m)$-system}
\begin{equation}\label{eq-wgdm}
\gd^*_{n,m}(L): \overline{\sigma}_{m}+\sum_{j\in J^*_{n,m-1}} c_j \overline{\sigma}_j=0,
\end{equation}
which is a modified GD system obtained by removing at each iteration $i$ the indices of the classes $[M_k]_n$, $M_k<m$, with
\begin{equation}
J^*_{n,m-1}=\bbZ_{m-1}\backslash \cup_{M_k<m} [M_k]_n.
\end{equation}

\begin{theorem}\label{thm:widehat}
Let $L$ be a differential operator in $\Sigma[\partial]$, of order $n$ and in normal form.
There exists a unique basis
\begin{equation}\label{eq-Gbasis}
\cG^*(L)=\{1, G^*_1, \ldots , G^*_{d-1}\}
\end{equation}
of $\centr (L)$ as a $\coC [L]$-module satisfying the following conditions:
\begin{enumerate}
    \item The orders $M_i=\ord(G^*_i)$ form an strictly increasing sequence
\[0<M_1<\cdots <M_{d-1}.\]
    \item For each $i\in \{1,\ldots , d-1\}$
\begin{equation}
 \frac{\centr_{M_i}(L)}{\centr_{M_i-1}(L)}\simeq \coC G^*_i.
\end{equation}
\item $G^*_i$ is defined by the unique solution  $\bc^*_{M_i}=(\bc^*_{M_{i},j}\mid j\in J^*_{n,M_{i}-1})$
 in $\coC^{| J^*_{n,M_{i}-1}|}$ of the system $\gd^*_{n,M_i}(L)$, namely
\begin{equation}\label{eq-hatG}
G^*_i=B_{M_i}+\sum_{j\in J^*_{n,M_{i}-1}} \bc^*_{M_i,j} B_j.
\end{equation}
\end{enumerate}
We will call $\cG^*(L)$ \sf{the filtered basis of $\centr (L)$}.
\end{theorem}
\begin{proof}
By Lemma \ref{lem-flag} then
    \[V_{M_i}=V_{M_i-1}\oplus T_i,\]
where $T_i$ is a one dimensional $\coC$-vector space $T_i=\langle (\bc_{M_i},1) \rangle$. Observe that $T_i$ is not unique and can be replaced by any $\langle (\bc_{M_i}+\xi,1) \rangle$ with $\xi\in V_{M_i-1}$. By Lemma \ref{lem-directsum}
\begin{equation}\label{eq-cenMi}
   \centr_{M_i-1}(L) = \bigoplus_{k=0}^{i-1} \coC_{M_i-M_k-1}[L] G_k,
\end{equation}
and by Lemma \ref{lem-quotient}  then
\begin{equation}
    \frac{V_{M_i}}{V_{M_i-1}}\simeq \frac{\centr_{M_i}(L)}{\centr_{M_i-1}(L)}=\langle G_i+ \centr_{M_i-1}(L)\rangle=\langle G^*_i+ \centr_{M_i-1}(L)\rangle,
\end{equation}
with $G^*_i$ as in \eqref{eq-hatG}. If there exists $\xi\neq \bc^*_{M_i}$ that is also a solution of $\gd^*_{n,M_i}(L)$ then there exists a solution of $\gd_{n,m}(L)$ for some $m\in J^*_{n,M_i-1}$, which is a contradiction by \eqref{eq-cenMi}.
\end{proof}

\begin{example}
    Let us assume that $n=5$ and $M_1=6$. Then
    $$G_1=B_6+\bc_{6,4} B_4+\bc_{6,3} B_3+\bc_{6,2} B_2+\bc_{6,1} B_1$$
    with $\bc_6=(\bc_{6,4},\bc_{6,3},\bc_{6,2},\bc_{6,1})$. Observe that $V_6=\langle (\bc_6,1) \rangle$ and $V_5=\{\overline{0}\}$.

    Let us also assume that $M_2=7$ then $V_7=V_6\oplus T_7$ with $T_7=(\bc_7,1)$ and  $\bc_7=(\bc_{7,6},\bc_{7,4},\bc_{7,3},\bc_{7,2},\bc_{7,1})$.
    Let
    $$G_2=B_7+\bc_{7,6} B_6+\bc_{7,4} B_4+\bc_{7,3} B_3+\bc_{7,2} B_2+\bc_{7,1} B_1.$$
    If $\bc_{7,6}\neq 0$ define $G^*_2=G_2
    -\bc_{7,6} G_1$ and observe that
    $$\bc^*_7=(\bc_{7,4}-\bc_{7,6},\bc_{7,3}-\bc_{7,6},\bc_{7,2}-\bc_{7,6},\bc_{7,1}-\bc_{7,6})$$
    is the unique solution of the system $$\gd^*_{5,7}(L): \overline{\sigma}_7+ c_4 \overline{\sigma}_4+c_{3}\overline{\sigma}_3+c_{2}\overline{\sigma}_2+c_{1}\overline{\sigma}_1=\overline{0}.$$
\end{example}

\subsection{Filtered basis algorithm}\label{sec:bounds}

We are ready to write an algorithm to compute the filtered basis of $\centr (L)$. Given $L$ in $\Sigma[\partial]$, of order $n$ and level $M$, with $\rho=(n,M)$, then by Theorem \ref{thm-Maxbasis} the elements of any basis $\cG_{\rho}$ of the maximal $\coC [L]$-submodule $\cM_{\rho}$ of rank $\rho$ verify
\[\ord(G)<B=(d-1) M, \mbox{ with }d=\frac{n}{\rho}.\]
The next algorithm computes a $\coC$-basis of
\[\widehat{\centr}_B(L):=\frac{\centr_{B}(L)}{\coC_{B}[L]}\]
and as a consequence a $\coC[L]$-basis of $\cM_{\rho}$. In particular, if $\rho=1$, it computes the filtered basis of  $\centr (L)$.

\begin{algorithm}[!ht]
\caption{\texttt{Filtered\_basis}}\label{alg:graded_basis}
    \Input{$L$ in $\Sigma[\partial]$, of order $n$ and in normal form; the level $M$ of $L$.}
    \Output{A $\coC$-basis of $\widehat{\centr}_B(L)$.}

    $B \gets (n/(n,M)-1)M$\;
    $\text{found} \gets \{[0]_n\}$\;
    $m\gets M$\;

    \For{$m=M,\ldots,B$}{%
        \If{$[m]_n \notin \text{found}$}{%
            $\xi\gets \text{Solve } \gd^*_{n,m}(L)$ in $\coC^{|J^*_{n,m-1}|}$\tcp*{$\leq$ one solution}

            \If{$\xi$ exists}{%
                $\widehat{G}_{|\text{found}|} \gets \sum_{j\in J} \xi_j B_j + B_m$\tcp*{As in\eqref{eq-wgdm}}
                $\text{found} \gets \text{found} \cup \{[m]_n\}$\;
            }
        }
    }
    \Return $[1,G^*_1,\ldots,G^*_{|\text{found}|}]$\;
\end{algorithm}

\begin{rem}
    By Theorem \ref{thm:widehat}, for every $m\notin [0]_n$ the optimized GD system $\gd^*_{n,m}(L)$ either has no solution or a unique solution until the next $M_i$ is found.
\end{rem}

\begin{theorem}
Consider $L \in \Sigma [\partial]$ of order $n$ and level $M$.
\begin{enumerate}
    \item  If $(n,M)=1$ then Algorithm \ref{alg:graded_basis} computes the filtered basis of the centralizer $\centr (L)$.

    \item  If $\rho=(n,M)\neq 1$ then Algorithm \ref{alg:graded_basis} computes a basis of a $\coC[L]$-submodule $\cM$ of $\centr (L)$ that contains the maximal submodule $\cM_{\rho}$ of $\centr (L)$,
    \[\cM_{\rho}\subseteq \cM \subseteq \centr (L).\]
\end{enumerate}
\end{theorem}
\begin{proof}
    By Lemma \ref{lem-Grho} and Theorem \ref{thm:widehat},  Algorithm~\ref{alg:graded_basis}  computes the filtered basis $\cG^*(L)=\{1, G^*_1, \ldots , G^*_{n-1}\}$ of $\centr (L)$ as a $\coC [L]$-module. Otherwise it computes a basis of the $\coC [L]$-module
    \[\cM=\oplus_{k=0}^{|\text{found}|} \coC[L] G^*_k.\]
\end{proof}

\begin{rem}
Given the basis $\cB$ computed in Algorithm \ref{alg:graded_basis}, $\cM$ has rank
\[\rank(\cM)=\gcd\{\ord(G)\mid G\in \cB\}.\]
Observe that:
\begin{enumerate}
    \item  $|\text{found}|<n-1$ because once $|\text{found}|=n$ the instruction in the loop is not executed. The number of systems $\gd^*_{n,m}(L)$, as in \eqref{eq-wgdm}, to be solved is at most $(d-1) M$, with $d=n/(n,M)$ but in practice it may be much smaller, since the procedure will stop once $G^*_{n-1}$ has been found. Note the examples in Section \ref{sec:families}.

    \item $\cM$ may not be a maximal submodule of $\rank(\cM)$.

    \item $\cB$ contains a basis of $\cM_{\rho}$ and possibly a set of extra operators $\cE$ that it is worth to analyze. If $\cE$ contains an operator $G^*$ whose order $M^*$ is coprime with $n$, then we can replace the bound $B$ by $B^*=(n-1)M^*$ and continue executing the while loop in Algorithm \ref{alg:graded_basis} to reach a basis of $\centr (L)$.
\end{enumerate}
\end{rem}

The key to implement Algorithm~\ref{alg:graded_basis} is to work in a field $\Sigma$ where the instruction of line 6 can be performed, namely where constant solutions of systems of linear equations with coefficients in $\Sigma$ can be computed. This is the goal of Section \ref{sec:computable}.

\section{Linearly computable differential fields}\label{sec:computable}

In  Section~\ref{sec:flag}, we developed Algorithm~\ref{alg:graded_basis} to compute the filtered basis $\cG^*(L)$ of the centralizer $\centr(L)$ as a $\coC[L]$-module. For this purpose we studied the linear systems ~\eqref{equ:gd}, derived from the Gelfand-Dickey hierarchies, and the flag of constants $V_m$ obtained as the constant solutions of the homogeneous systems ~\eqref{equ:hgd}. The algorithm is based on computing the constant solutions of the optimized systems $\gd^*_{n,m}(L)$ in \eqref{eq-wgdm} and this is the purpose of this section.

It is important to remark that the linear systems  $\gd_{n,m}(L)$ in ~\eqref{equ:hgd} and $\gd^*_{n,m}(L)$ in \eqref{eq-wgdm}, have their coefficients in the differential field $\Sigma$, the differential field of coefficients of our differential operator $L$. On the other hand, the solutions we are looking for are in $\coC^{|J_{n,m}|}$ (instead of the usual $\Sigma^{|J_{n,m}|}$). Hence, classical algorithms to compute solutions of linear systems are not enough.
In this section, we are going to study a particular case where, with some structural information about the differential field $\Sigma$, we can actually perform all the necessary operations to precisely compute the solutions in $\coC^{|J_{n,m}|}$ of our linear system.

We propose to study computable differential fields, where all operations can be performed algorithmically. In particular, we are interested in those that have a zero-testing algorithm that is linear in some sense. The simplest computable differential fields of this kind that we encounter are the so-called monomial extensions \cite{Bron},  differential extensions $\coC\langle\eta\rangle$ by a transcendental element $\eta$ whose derivative belongs to $\coC\langle\eta\rangle$.

\begin{definition}
    Let $(\Sigma, \partial)$ be a differential field with field of constants $\coC$.
    A {\sf constant solution of a linear system} $S$, with coefficients in $\Sigma$ and $s$ unknowns, is a solution in $\coC^s$.
    We say that $\Sigma$ is {\sf linearly-computable} if it is computable and there is an algorithm that for any linear system with coefficients in $\Sigma$ computes all the constant solutions.
\end{definition}

In this paper, motivated by the main solutions of the KdV hierarchy \cite{Veselov}, we are going to consider a special type of linearly-computable differential extension of the field of constants $\coC$, although similar constructions and ideas could be applied to any linearly-computable differential field.

\begin{lemma}\label{lem:def_elliptic}
    Let us consider $(\coC,0)$ be a differential field, $\eta$ a transcendental element over $\coC$, and $\nu$ an algebraic element over $\coC(\eta)$. Then, there is a unique derivation $\partial$ such that  $(\coC(\eta)[\nu],\partial)$, with $\nu=\partial(\eta)$, is a differential extension of $(\coC, 0)$.
\end{lemma}
\begin{proof}
    First, we need to remark that $\partial(\eta)$ determines the value of $\partial(\nu)$. Since $\nu$ is algebraic over $\coC(\eta)$, there is a monic irreducible polynomial $p(y) \in \coC(\eta)[y]$ such that $p(\nu) = 0$. Hence, for any derivation over $\coC(\eta)[\nu]$ it holds that $\partial(p(\nu)) = 0$. Similar to Theorem~3.2.3 from~\cite{Bron}, there is a polynomial $s(y) \in \coC(\eta)[y]$ such that $s(\nu)\partial_y(p)(\nu) = 1$, so we have
    \[\partial(\nu) = -s(\nu)\kappa_\partial(p)(\nu),\]
    where $\kappa_\partial$ is the derivation on $\coC[\eta,y]$ defined by $\kappa_\partial(\eta) = \partial(\eta)$ and $\kappa_\partial(y) = 0$.

    Then, we can show that, for any $f(\eta,\nu) \in \coC(\eta)[\nu]$, we can write
    \[\partial(f(\eta,\nu)) = \kappa_\partial(f)(\eta,\nu) + \partial(\nu)\partial_y(f)(\eta,\nu).\]
    Hence, given $\partial(\eta)$, we have defined the values for $\partial(f(\eta,\nu))$ for all $f(\eta,\nu) \in \coC(\eta)[\nu]$. In particular, there is a unique derivation extending the zero derivation over $\coC$ such that $\partial(\eta) = \nu$.
\end{proof}
\begin{definition}\label{def-elliptic}
Consider a transcendental element $\eta$ over $\coC$, and $\nu$ an algebraic element over $\coC(\eta)$ of degree at most two, i.e. $\nu^2 = \gamma_0(\eta)+\gamma_1(\eta)\nu$ for an irreducible polynomial $y^2 -\gamma_1(\eta)y-\gamma_0(\eta)$ with coefficients in $\coC(\eta)$. We call {\sf hyperelliptic extension} of $\coC$ by $\eta$, and denote it by $\coC\langle\eta\rangle$, to the differential field extension $(\coC(\eta)[\nu],\partial)$ where $\partial(\eta) =\nu$.
\end{definition}

This definition is satisfied by some elementary functions over $\coC=\bbC$ taking $\partial=d/dx$. For example, if we consider $\eta=e^x$, then $\eta' = e^x$ is in $\bbC(\eta)$. In the case of hyperbolic trigonometric functions, such as $\eta=\cosh(x)$, we have $\eta'=\sinh(x)$ with the identity $\eta'^2 = 1-\eta^2$. Finally, the Weierstrass $\wp(x,g_2,g_3)$, $g_2,g_3\in \bbC$, function also falls into this setting, since we know that $\wp'^2=\wp^3+g_2\wp+g_3$.

\begin{proposition}\label{prop:diff_computable}
    Let $\coC\langle \eta\rangle$ be a hyperelliptic differential extension of $\coC$. Then $\coC\langle \eta\rangle$ is linearly-computable.
\end{proposition}
\begin{proof}
An hyperelliptic differential extension is defined by an irreducible polynomial $y^2-\gamma_1(\eta)y-\gamma_0(\eta)$ with coefficients in $\coC(\eta)$.
Thus $\coC\langle \eta\rangle$ is isomorphic to the fraction field of the domain
$\frac{\coC(\eta)[y]}{(y^2-\gamma_1(\eta)y-\gamma_0(\eta))}$.
It is also computable as a field. Moreover, we can use $y^2-\gamma_1(\eta)y-\gamma_0(\eta)=0$ to compute the derivative of $\eta'$, hence we can compute the derivative of any element in $\coC\langle \eta\rangle$.

Now, let $S = \{e_i = 0\mid i=1,\ldots,r\}$ be a linear system with coefficients in $\coC\langle\eta\rangle$ and unknowns $a_1, \ldots ,a_s$. Let $$e_i = \alpha_1 a_1  + \ldots + \alpha_s a_s , \mbox{ where } \alpha_k = \frac{p_{k,1}(\eta)}{q_{k,1}(\eta)}\eta' + \frac{p_{k,0}(\eta)}{q_{k,0}(\eta)},$$
    where $p_{k,j}$ and $q_{k,j}$ are polynomials in $\coC[\eta]$.
    Using the fact that $\coC\langle\eta\rangle[a_1,\ldots,a_s] \simeq \coC[a_1,\ldots a_s]\langle\eta\rangle$, we can rewrite equation $e_i$ as a polynomial in $\eta'$ whose coefficients are rational functions in $\coC[a_1,\ldots,a_s](\eta)$:
    \[e_i = \frac{\Delta_0}{\Lambda_0} +  \frac{\Delta_1}{\Lambda_1}\eta'\]
    where $\Lambda_j=\text{lcm}(q_{k,j})$, $D_{k,j}=\Lambda_j q_{k,j}$ and
    \[
    \Delta_j=\sum_{k=1}^s D_{k, j} p_{k,j}(\eta) a_k,\,\,j=0,1.
    \]
    Computing the constant solutions in $\coC^s$ of $e_i = 0$ is equivalent to equating $\Delta_0(\eta) = 0$ and $\Delta_1 (\eta) = 0$. Since $\eta$ is transcendental over $\coC$, this is equivalent to equating all the coefficients of the monomials $\eta^h$ in $\Delta_j$ to zero, which are linear polynomials in $\coC[a_1,\ldots,a_s]$. Gathering all the coefficients of $\eta^h$ in $\Delta_0$ and $\Delta_1$ for each $e_i$, we obtain a system $\tilde{S}$ of linear equations in $a_k$ with coefficients in $\coC$.
We conclude that the constant solutions of the linear system $S$ are the  solutions of the extended system $\tilde{S}$.
\end{proof}

Let us denote by $\mathbb{M}_{r\times s}(K)$ the ring of matrices with coefficients in a field $K$.
We represent a system $S$ with $r$ linear equations, with coefficients in $\coC\langle \eta\rangle$, and $s$ unknowns, by a matrix in $\mathbb{M}_{r\times s}(\coC\langle \eta\rangle)$. The goal of Algorithm~\ref{alg:extend_system} is to obtain an extended linear system $\tilde{S}$ with coefficients in $\coC$, whose solution set equals the set of constant solutions of $S$, the solutions in $\coC^s$ of $S$.

Algorithm~\ref{alg:extend_system} is based on collecting the coefficients of a polynomial in the variable $\eta$. Given $p\in \coC [\eta]$ we denote by $\coeff(p,\eta^h)$ the coefficient of $\eta^h$ in $p$. Each equation $e$ of the system $S$ can be written as
\begin{equation}\label{equ:equations}
    e=\frac{\Delta_0}{\Lambda_0} + \frac{\Delta_1}{\Lambda_1} \eta',
\end{equation}
where $\Delta_j$ are polynomials in $\eta$, whose sets of non-zero coefficients $T_{e,j}$ form the linear equations of the extended system $\tilde{S}$.
Observe that the total number of equations of $\tilde{S}$ is
\begin{equation}\label{equ:monomials}
    T=\sum_{e\in S} (|T_{e,0}|+|T_{e,1}|).
\end{equation}

\begin{algorithm}[!ht]
\caption{\texttt{Extend\_system}}\label{alg:extend_system}
    \Input{Hyperelliptic differential extension $\coC\langle \eta\rangle$ of $\coC$; linear system $S \in \mathbb{M}_{r\times s}(\coC\langle \eta\rangle)$ with $s$ unknowns}
    \Output{Extended linear system $\tilde{S}\in \mathbb{M}_{T\times s}(\coC)$ whose solution set is the set of constant solutions of $S$}

    $\tilde{S} \gets []$\;
    \For{$e$ row of $S$}{%
        Assume $e = (\alpha_1,\ldots,\alpha_s)$ with $\alpha_i = \frac{p_{i,1}(\eta)}{q_{i,1}(\eta)}\eta' + \frac{p_{i,0}(\eta)}{q_{i,0}(\eta)}$\;
        $D_{i,j} = \text{lcm}(q_{k,j}(\eta)\ :\ k=1,\ldots,s) / q_{i,j}(\eta)$\;
        $T_{e,j} \gets \{h \in \mathbb{N}\ \mid\ \coeff(p_{i,j}(\eta) D_{i,j}, \eta^h)\neq 0
        \text{ for some }i=1,\ldots,s\}$\;
        \For{$j=0,1$}{%
            \For{$h \in T_{e,j}$}{%
                Append row
                $\left(\coeff(p_{i,j}(\eta) D_{i,j}, \eta^h) \ :\ i=1,\ldots,s\right)$
                to $\tilde{S}$\;
        }}
    }
    \Return $\tilde{S}$\;
\end{algorithm}

Algorithm~\ref{alg:get_system} builds on Algorithm~\ref{alg:extend_system} to define the linear systems to solve in order to compute elements of the centralizer. More precisely, if we provide a finite set $J$ of positive integers, then any solution to the inhomogeneous linear system returned by Algorithm~\ref{alg:get_system} corresponds to an element of the centralizer $\centr(L)$ of order exactly the maximum $\max(J)$ of $J$,
\begin{equation}
    F_J:=\{\xi\in \coC^{|J|}\mid \xi \mbox{ is a solution of } \overline{\sigma}_{m}+\sum_{j\in J} c_j \overline{\sigma}_j=0,\ m= \max(J)\}.
\end{equation}
Of course, it may happen that $F_J$ is the empty set.
If $J = J_{n,m}$, then the system is $\gd_{n,m}(L)$ and its solution set is precisely $F_m$. In Algorithm \ref{alg:goodearl_basis} we will consider $J = J^*_{n,m}$, in which case Algorithm \ref{alg:get_system} builds the system $\gd^*_{n,m}(L)$ in \eqref{eq-wgdm}.

\begin{algorithm}[!ht]
\caption{\texttt{Build\_system}}\label{alg:get_system}
    \Input{Hyperelliptic differential extension $\coC\langle \eta\rangle$ of $\coC$; $\upsilon_L=\{\upsilon_2,\ldots,\upsilon_n\}\subset \coC\langle \eta\rangle$; finite set $J \subset \bbN$}
    \Output{A linear system $(A|b)\in \mathbb{M}_{T\times |J|}(\coC)$, whose solution set is $F_J$
}

    $\sigma_{k,j} \gets \cE_L (H_{k,j})$\tcp*{  specialization in~\eqref{eq-sigmakj}}
    $S \gets \left(\sigma_{k,j}\ \mid\ k=0,\ldots,n-2,\quad j\in J\right) \in \mathbb{M}_{(n-1)\times |J|}(\coC\langle \eta\rangle)$\;
    $\tilde{S} \gets \texttt{Extend\_system}(S)$ \tcp*{$\tilde{S} \in \mathbb{M}_{ T\times |J|}(\coC)$}
    $A \gets$ first $|J|-1$ columns of $\tilde{S}$\;
    $b \gets$ last column of $\tilde{S}$\tcp*{last column is linked to $\max(J)$}
    \Return $(A|-b)$\;
\end{algorithm}

With Algorithm~\ref{alg:get_system}, we have all the necessary tools to apply all the results of Section~\ref{sec:flag} in order to implement Algorithm~\ref{alg:graded_basis}. We rewrite the algorithm in Algorithm  ~\ref{alg:goodearl_basis}, \texttt{Filtered\_basis\_computable} assuming now that the coefficients $\upsilon_L=\{\upsilon_2,\ldots,\upsilon_n\}$ of $L$ belong to an hyperelliptic differential field $\coC\langle \eta\rangle$. We can now replace instruction 6 of Algorithm~\ref{alg:graded_basis} by a call to \texttt{Build\_system} in Algorithm ~\ref{alg:goodearl_basis}.

\begin{algorithm}[!ht]
\caption{\texttt{Filtered\_basis\_computable}}\label{alg:goodearl_basis}
    \Input{Hyperelliptic differential extension $\coC\langle \eta\rangle$ of $\coC$; $\upsilon_L=\{\upsilon_2,\ldots,\upsilon_n\}\subset \coC\langle \eta\rangle$; the level $M$ of $L$; bound $B= (n-1)M$ if $(n,M)=1$ or $B\geq (n-1)M$ otherwise}
    \Output{Filtered basis of $\centr_B (L)/\coC_B[L]$}

    $\text{found} \gets \{[0]_n\}$\;

    \For{$m=M,\ldots,B$}{%
        \If{$[m]_n \notin \text{found}$}{%
            $(A|b) \gets \texttt{Build\_system}(\upsilon_L, J^*_{n,m})$\;
            $\xi \gets \texttt{solve}(A, b)$\tcp*{Solving the linear system in $\coC$}
            \uIf{$\xi\text{\emph{ \texttt{"exists"}}}$}{%
                $G^*_{|\text{found}|} \gets \sum_{j\in J} \xi_j B_j + B_m$\tcp*{As in\eqref{eq-wgdm}}
                $\text{found} \gets \text{found} \cup \{[m]_n\}$\;
            }
        }
    }
    \Return $[1,G^*_k, k=1,\ldots ,|\text{found}|]$\; 
\end{algorithm}

\begin{rem}
If $(n,M)=1$, the search bound is $B=(n-1)M$ and output is the filtered basis $\{1,G^*_1,\ldots,G^*_{n-1}\}$ of $\centr(L)$ as a $\coC [L]$-module.

If $(n,M)\neq 1$, the search bound can be chosen $B\geq (n-1)M$, normally equal.
\end{rem}

\begin{example}[Continuation of Example~\ref{exm:hierarchy:3:4:eval}, Hyperbolic]\label{exm:centralizer:3:4:hyperbolic}
Since the field $\Sigma=\bbC\langle \cosh(x)\rangle$ is a hyperelliptic differential extension we can apply the previous algorithms. Let us run Algorithm~\ref{alg:goodearl_basis} with $\upsilon_L=\{v_2=6, v_3=0\}$, $M=4$ and $B=(n-1) M=8$. To start $\text{found}=\{[0]_3\}$.

We start the for loop in Step 3 with $m=4$ and call $\texttt{Build\_system}(\upsilon_L, J^*_{3,4})$, with $J^*_{3,4}=\{1,2,4\}$.
In Example ~\ref{exm:hierarchy:3:4:eval}, Hyperbolic, denoting $\eta=\cosh(x)$, we computed $\hgd_{3,4}(L)$ with coefficient matrix $S(3,4)$ whose columns are indexed by $c_1$, $c_2$ and $c_4$
    \[S(3,4)=\begin{pmatrix}
0 & \displaystyle\frac{48-32\eta^2}{\eta^5}\eta' & \displaystyle\frac{192-128\eta^2}{3\eta^5}\eta'\\
\displaystyle\frac{12 \eta'}{\eta^3} & \displaystyle\frac{24\eta^2 - 36}{\eta^4} & \displaystyle\frac{32\eta^2 - 48}{\eta^4}
\end{pmatrix}.\]
Let us build the extended linear system $\tilde{S}(3,4)$ with coefficients in $\bbC$.
For each row, Algorithm~\ref{alg:extend_system} clears denominator and collects the elements with the same monomials. As an example, take the second row of $S(3,4)$, multiplying by $\eta^4$ we obtain
    \[\left(12\eta'\eta, 24\eta^2 - 36, 32 \eta^2 - 48\right).\]
Collecting coefficients of the monomials $\eta\eta'$, $\eta^2$ and $1$ we obtain the last three rows of
    \[\tilde{S}(3,4)=\begin{pmatrix}
        0 & 48 & 64\\
        0 & -32 & -\frac{128}{3}\\
        12 & 0 & 0 \\
        0 & 24 & 32 \\
        0 & -36 & -48
    \end{pmatrix}\begin{array}{l} \\   \\\rightarrow \eta\eta' \\ \rightarrow \eta^2 \\ \rightarrow 1
    \end{array}\]
Solving the system in $c_1$, $c_2$ and $c_4$ with this coefficient matrix we obtain $V_4=\langle (0, \frac{-4}{3},1)\rangle$, the $\coC$-vector space of constant solutions of $\hgd_{3,4}$ is, see \eqref{eq-Vm}.
The first iteration of
Algorithm~\ref{alg:goodearl_basis}  returns the operator
    \[G^*_1=B_4-\frac{4}{3}B_2= \partial^4 +\left(\frac{24 -4 \eta^{2}}{3 \eta^{2}}\right)\partial^2+ \left(\frac{-8\eta'}{\eta^{3}} \right)\partial,\]
and sets $\text{found}=\{[0]_3, [1]_3\}$.

In the second iteration with $m=5$ and call $\texttt{Build\_system}(\upsilon_L, J^*_{3,5})$, with $J^*_{3,5}=\{1,2,5\}$. Since we have found an element in the centralizer of order $4$, we are going to look for a linear combination $c_1 B_1 + c_2 B_2 + c_5B_5$, to obtain the filtered basis. The system $\gd^*_{3,5}(L)$ has  coefficient matrix $S(3,5)$ with columns indexed by $c_1,c_2,c_5$ and extending this matrix with Algorithm~\ref{alg:extend_system} we obtain $\tilde{S}(3,5)$:
\[S(3,5)=\begin{pmatrix}
0 & \displaystyle\frac{48-32\eta^2}{\eta^5}\eta' & 0\\
\displaystyle\frac{12 \eta'}{\eta^3} & \displaystyle\frac{24\eta^2 - 36}{\eta^4} & \displaystyle\frac{-64\eta'}{3\eta^3}
\end{pmatrix},\quad \tilde{S}(3,5) = \begin{pmatrix}0 & 48 & 0 \\
0 & -32 & 0 \\
12 & 0 & -\frac{64}{3} \\
0 & -36 & 0 \\
0 & 24 & 0\end{pmatrix}.\]
This matrix represents the homogeneous system, whose solutions are of the form $\tilde{V}_5 = \langle(16/9,0,1)\rangle$. Hence
\begin{align*}
    G^*_2 = B_5+\frac{16}{9}B_1 & = \partial^5 + \frac{10}{\eta^2}\partial^3 - 20\frac{\eta'}{\eta^3}\partial^2
     + \left(\frac{16\eta^4+240\eta^2-180}{9\eta^4}\right)\partial.
\end{align*}
Now $\text{found}=\{[0]_3, [1]_3, [2]_3\}$.
This concludes the execution of Algorithm~\ref{alg:goodearl_basis}, since from $m=6,\ldots ,B=8$, $[m]_3\in \text{found}$.
Namely, the filtered basis of the centralizer $\centr (L)$ is $\{1,G^*_1,G^*_2\}$, where the orders of the operators are $M_0=0$, $M_1=4$ and $M_2=5$.
\end{example}

\begin{example}\label{exm:Zheglov}\cite[Example 22]{PZ}
    In~\cite{PZ}, an algorithm is designed, using the Sato operator, to compute pairs of commuting differential operators of a given rank. In Example~22, a pair of ODOs of orders $4$ and $6$ with rational coefficients is provided. Let us consider the operator $L$ of order $4$
\begin{align*}
\partial^4- 20 \frac{x^3(x^5-120)}{
(x^5 + 30)^2} \partial^2-3000
\frac{x^2(7x^5-90)}{(x^5 + 30)^3} \partial + 18000
\frac{x(3x^{10}-145x^5 + 450)}{
(x^5 + 30)^4}.
\end{align*}
    From this information we cannot determine the rank of the centralizer $\centr (L)$, we can only say that is less than or equal to $2$, see Remark \ref{rem-Zheglov}.

    We checked that $L$ has level $M=6$, i.e. there is not a non-trivial element in $\centr(L)$ of order smaller than 6. We can do so by running the for loop in Algorithm ~\ref{alg:goodearl_basis}, for $m=0,\ldots ,6$. By means of Algorithm~\ref{alg:goodearl_basis} we can prove that $\rank (\centr(L))=1$. We run Algorithm~\ref{alg:goodearl_basis} with $B=(n-1) M=18$ and obtain the filtered basis $\{1, G_1^*, G_2^*,G_3^*\}$ as a $\coC[L]$-module:
    \begin{itemize}
        \item $G_1^*$ has order $M_1=6$ coincides with the operator $P_6$ in ~\cite[Example 22]{PZ}.
        \item $G_2^*$ is an operator of order $M_2=7$. The coefficients of this operator are quotients of polynomials of degree bounded by $23$.
        \item $G_3^*$ is an operator of order $M_3=9$, whose coefficients are quotients of polynomials of degree bounded by $31$.
    \end{itemize}
    We refer to Section~\ref{sec:implementation} for indications on how to access the precise values of these operators.
\end{example}

\section{Level varieties}\label{sec-levelvar}

In this section we will design an algorithm to compute families of differential operators with non trivial centralizers, assuming that their coefficients belong to a fixed differential field dependent on a finite set $\Theta$ of algebraic parameters over the field of constants $\coC$.

Let us consider the field $\coC(\Theta)$ of rational functions in $\Theta$ to define a hyperelliptic differential extension $\bbK=\coC(\Theta)\langle \eta\rangle$ of $\coC(\Theta)$, as in Definition \ref{def-elliptic}, with the extended derivation $\partial$ such that $\partial \theta=0$, for every $\theta\in \Theta$. Given a fixed set of functions $f_2,\ldots, f_n \in \bbK$, we consider the following ansatz operator
\begin{equation}\label{eq-bbL}
\bbL= \partial^n + f_2\partial^{n-2}+\ldots+f_n.
\end{equation}
Our strategy will be to fix $M\in \bbN$ and obtain polynomial relations of the variables in $\Theta$ that force $\bbL$ to have level $M$, obtaining a family of operators $\bbL_M$ of level $M$.

Given $\zeta\in \coC^{|\Theta|}$ and replacing $\Theta$ by $\zeta$ in $\bbL$ we obtain a differential operator $\bbL(\zeta)$ in $\coC\langle \eta\rangle[\partial]$.
We know that $\bbL(\zeta)$ has level $M$ if the system $\gd_{n,M}(\bbL(\zeta))$ has a unique solution, by Lemma \ref{lem-flag}, if $F_M\neq \emptyset$, see~\eqref{eq-Fm}. Using  Algorithm~\ref{alg:get_system} with $\upsilon_{\bbL}=\{f_2,\ldots ,f_n\}\subset \bbK$ and $J=J_{n,M}$, we obtain a linear system with coefficient matrix $S_{\Theta}(n,M)= (A|b)$ with entries in $\coC (\Theta)$. This is a key fact in the development of the next algorithm since $F_M\neq \emptyset$ whenever the system $S_{\Theta}(n,M)$ has a unique solution $(\xi,1)$ with $\xi\in \coC^{|J_{n,M-1}|}$. However, imposing a unique solution involves a set of inequalities.
On the other hand, studying the set  $V_M(\bbL(\zeta))$  of constant solutions of $\hgd_{n,M}(\bbL(\zeta))$ as in \eqref{eq-Vm}, and imposing a non-zero solution is easy to handle algebraically.

\begin{definition}
Given $m\in\bbN$, we define the $m$-th {\sf level variety} of $\bbL$ to be
\[\cX_m=\{\zeta\in \coC^{|\Theta|} \mid V_m(\bbL(\zeta))\neq \{0\}\}.\]
\end{definition}

To impose the condition
$V_m(\bbL(\zeta))\neq \{0\}$, the rank of the matrix $S_{\Theta}(n,m)$ must be smaller than the number of variables $|J_{n,m}|$, assuming that the number of rows $T$ is greater than $|J_{n,m}|$. That is the case when all minors of size $|J_{n,m}|$ of the matrix $S_{\Theta}(n,m)$ must vanish. These minors are rational functions in $\coC(\Theta)$ and they all must be zero at the same time.

\begin{definition}
    Given $m\in\bbN$, we define the $m$-th {\sf level ideal } of $\bbL$ to be the ideal $I_m$ in $\coC[\Theta]$  generated by the numerators of the minors of size $|J_{n,m}|$ of $S_{\Theta}(n,m)$, assuming that $T\geq |J_{n,m}|$.
\end{definition}

\begin{lemma}\label{lem-chainvar}
Given $\bbL$ as in \eqref{eq-bbL}, then $\cX_m $ is the algebraic variety defined by the level ideal $I_m$
\[\cX_m=V(I_m)=\{\zeta\in \coC^{|\Theta|} \mid g(\zeta)=0, \forall g\in I_M\}.\]
Moreover, $\cX_{m-1}\subseteq \cX_m$, for every $m\geq 1$.
\end{lemma}
\begin{proof}
    From the previous discussion and the fact that $V_{m-1}(\bbL(\zeta)) \subseteq V_m(\bbL(\zeta))$ by Lemma \ref{lem-flag}.
\end{proof}

As a consequence of Lemma \ref{lem-chainvar}, there is a chain of algebraic varieties in $\coC^{|\Theta|}$
\[\cX_0 \subseteq \cX_1 \subseteq \cX_2 \subseteq \ldots \subseteq \cX_M \subseteq \ldots,\]
and by the ideal-variety correspondence \cite{Cox}, a chain of radical ideals
\[\sqrt{I_0}\supseteq \sqrt{I_1}\supseteq \sqrt{I_2} \supseteq \ldots \supseteq \sqrt{I_n} \supseteq \ldots.\]
Moreover, if $\zeta \in \coC^{|\Theta|}$ belongs to $\cX_M\setminus \cX_{M-1}$, then replacing $\Theta$ by $\zeta$ in $\bbL$ we obtain a differential operator $\bbL(\zeta)$ in $\coC\langle \eta\rangle[\partial]$, that has level $M$. Algorithm~\ref{alg:get_ideal} summarizes the steps to compute the variety $\cX_M$ for a fixed value of $M$ and a fixed template of parametric functions $f_2,\ldots,f_n$.

\begin{algorithm}[!ht]
\caption{\texttt{Level\_Variety}}\label{alg:get_ideal}
    \Input{Hyperelliptic differential extension $\coC\langle \eta\rangle$ of $\coC$; a set $\Theta$ of algebraic parameters; $\{f_2,\ldots,f_n\}\subset \coC(\Theta)\langle \eta\rangle$; desired level $M$}
    \Output{Level variety $\cX_M$}

    $S \gets \texttt{Build\_system}(\{f_2,\ldots,f_n\}, J_{n,M})$\;
    $T \gets \mbox{number of rows of } S$\;
    \uIf{$T>|J_{n,M}|$}
    {
    $t \gets |J_{n,M}|$\;
    $I \gets \left(\text{ minors of size $t\times t$ of $S$}\right)$\;
    $\cX\gets V(I)$\;
    }

    \Return  $\cX$\;
\end{algorithm}

The family of algebro-geometric ODOs of level $M$ is defined by
\begin{equation}
   \bbL_M=\{\bbL(\zeta)\mid \zeta\in \cX_M\backslash \cX_{M-1}\}.
\end{equation}
Observe that it could be empty, finite or infinite.

\begin{rem}
   The last step of Algorithm  \ref{alg:get_ideal}, the computation of $\cX_M$ from $I_M$, involves solving a system of non-linear equations defined by polynomials in $\coC[\Theta]$. A study of the system of nonlinear equations defined by $I_M$ determines the family $\bbL_M$. We illustrate this analysis in examples \ref{exm:ideal:3:4:hyperbolic} and \ref{exm:ideal:4:6:rational}, see also Section \ref{sec:families}. We propose the next conjecture, that is left for future research in connection with factorization results of the operators in the family $\bbL_M$.

\medskip

   {\sf Conjecture.} If $(n,M)=1$ then $\cX_M\backslash \cX_{M-1}$ is a finite set of points, otherwise it is a variety of dimension greater than zero.

   \medskip

   If the variety is a finite set of points, its computation can be achieved using polynomial system solvers as
\texttt{HomotopyContinuation.jl} \cite{BT2018} or \texttt{msolve} \cite{BES2021}.
The use of such tools is also a future project.
\end{rem}

\begin{example}[ODOs of level $M=4$ and hyperbolic coefficients]\label{exm:ideal:3:4:hyperbolic}

    We study operators of the form
    $$\bbL =\partial^3 + f_2 \partial + f_3= \partial^3 + \frac{a_2}{\cosh(x)^2}\partial + \frac{a_3\sinh(x)}{\cosh(x)^3},$$
    with set of parameters $\Theta=\{a_2,a_3\}$.
    We apply Algorithm~\ref{alg:get_ideal} with field $\bbC(\Theta)\langle\eta\rangle$, $\eta=\cosh(x)$.
    Observe that the specialization map $$\cE_{\bbL}:\bbC\{u_2,u_3\}\rightarrow \bbC(\Theta)\langle\eta\rangle$$ defined by  $\cE_{\bbL}(u_2) = f_2$ and $\cE_{\bbL}(u_3) =f_3$ is used to build $\hgd_{3,4}(\bbL)$.
    From $\texttt{Build\_system}(\{f_2,f_3\}, J_{3,4})$ in Algorithm \ref{alg:get_system} we obtain $S_{\Theta}(3,4)$
\begin{equation*}
\left(
\begin{array}{ccc}
-3 a_3 & 0& 0\\
2 a_3 & 0& 0\\
0& \frac{-4}{3} a_2^2 + 16 a_2 + 12 a_3 & -16 a_2^2 - \frac{16}{3}  a_2 a_3 + \frac{8}{3} a_3^2 + \frac{320}{3} a_2 + 80 a_3\\
0& \frac{16}{3} a_2 - 4 a_3 &  \frac{-64}{9} a_2 - \frac{16}{3} a_3\\
0& 0 & \frac{-8}{9} a_2^3 + 32a_2^2 + 12 a_2 a_3 - 4 a_3^2 - 160 a_2 - 120 a_3\\
2 a_2& 0& 0\\
0& -6 a_2 - 6 a_3 & \frac{16}{3}a_2^2 + \frac{16}{3}a_2 a_3 - 40 a_2 - 40 a_3\\
0& 4 a_2 + 4 a_3& \frac{16}{3}a_2 + \frac{16}{3} a_3\\
0& 0& \frac{-20}{3} a_2^2 -\frac{20}{3}a_2 a_3 + 40 a_2 + 40 a_3
\end{array}
\right)
\end{equation*}
with 39 non-zero $3\times 3$ minors that generate the level ideal $I_4$. The Gröbner basis of $I_4$ for the degree-reverse lexicographic order of $a_2$ and $a_3$ is
    \[I_4 = \left(a_2^3(a_2-6),\quad a_3^3(a_3+12),\quad  a_3  (2a_2 - a_3)  (2a_2 + a_3),\quad (2a_2 + a_3)  a_3^2\right).\]
This is an easy ideal to analyze, leading to
$$\cX_4 = \{(0,0), (6,0), (6,-12)\}.$$
We also compute $\cX_3=\cX_2$.
The matrix $S_{\Theta}(3,2)$ is the submatrix of $S_{\Theta}(3,4)$ obtained by removing the last column and rows $5$ and $9$. The ideal $I_2$ is generated by all $2\times 2$ minors of the matrix $S_{\Theta}(3,2)$. The corresponding variety is $\cX_2=\{(0,0)\}$. Finally
\[\bbL_4=\{\bbL(6,0), \bbL(6,-12)\}.\]
In examples~\ref{exm:hierarchy:3:4:eval} (the hyperbolic case) and~\ref{exm:centralizer:3:4:hyperbolic} the operator $L=\bbL(6,0)$ is used.
\end{example}

\begin{example}[Level varieties for Rational coefficients]\label{exm:ideal:4:6:rational}
    Let us consider an operator of the form
    \[\bbL =\partial^4 + f_2 \partial^2 + f_3\partial + f_4=\partial^4 + \frac{a_2}{x^2} \partial^2 + \frac{a_3}{x^3}\partial + \frac{a_4}{x^4},\]
    with set of parameters $\Theta=\{a_2,a_3,a_4\}$.
    To analyze the level varieties for different level values  $M=1,\ldots,6$, for $M\neq 4$, we apply Algorithm~\ref{alg:get_ideal}. The call $\texttt{Build\_system}(v_{\bbL},J_{4,M})$ for $v_{\bbL}=\{f_2,f_3,f_4\}$ returns a matrix $S_{\Theta}(4,M)$ with very simple structure: $|J_{4,M}|$ columns and $3|J_{4,M}|$ rows that split into three diagonal blocks
\[S_{\Theta}(4,6)=\begin{pmatrix}
{\rm diag}(4 a_{4},3 a_{2}^{2} - a_{2} a_{3} + 60 a_{2} - 20 a_{4},s_{1,3}, s_{1,4},s_{1,5})\\
{\rm diag}(3 a_{3},-2 a_{2}^{2} - 48 a_{2} - 12 a_{3} + 8 a_{4},s_{2,3}, s_{2,4},s_{2,5})\\
{\rm diag}(2 a_{2},12 a_{2} + 6 a_{3},s_{3,3}, s_{3,4},s_{3,5})
\end{pmatrix},\]
where $s_{i,j}$ belong to $\bbC[\Theta]$.
By Lemma \ref{lem-chainvar} it is guarantied that we have a chain of level varieties
\[
\begin{cases}
\cX_1=\{(0,0,0)\},\\
\cX_2=\cX_1 \cup \left\{\left(\frac{-t}{2}, t, \frac{1}{16}t^2 - \frac{3}{2}t\right)\ :\ t \in \bbC\right\},\\
\cX_3 = \cX_2 \cup \left\{(-8, 8, 0), (-8, 24, -24), (-20, 40, 0), (-4, 8, -8)\right\},\\
\cX_5 =\cX_3 \cup  \left\{(-24, 24, 0), (-24, 72, -72), (-12, 48, -72), (-12, 0, 0), \right.\\
\left.(-36, 24, 144), (-36, 120, 0), (-60, 120, 216), (-28, 56, -56),\right. \\
\left. (-16, 32, 40), (-12, 24, 0)\right\}\}.
\end{cases}.
\]
It can be easily checked that the variety $\cX_2\backslash\cX_1$ are the conditions for $\bbL$ to factor as a square of an operator of order $2$. Finally, for $M=6$ we obtain
\[\cX_6 = \cX_5 \cup \left(\cX_{6,1} \cup \cX_{6,2}\right).\]
where $\cX_{6,1}$ and $\cX_{6,2}$ are varieties defined respectively by the next ideals
\begin{align*}
            I_{6,1} & {}= \left(4 a_{2} a_{3} + a_{3}^{2} + 24 a_{3} + 16 a_{4}, a_{2}^{2} + 32 a_{2} - 6 a_{3} - 4 a_{4} + 240\right),\\
            I_{6,2} & {}= \left(2 a_{2} + a_{3}, a_{3}^{2} - 152 a_{3} - 16 a_{4} + 4480\right)
        \end{align*}
and we could check that they define irreducible curves in $\bbC^3$.
\end{example}

\section{Families of algebro-geometric ODOs}\label{sec:families}

In this section we use Algorithm \ref{alg:get_ideal} to generate families of algebro-geometric differential operators with coefficients in a fixed differential field $\coC\langle \eta\rangle$, for different types of $\eta$ inspired from the solutions of the classical KdV hierarchy. We combine it with Algorithm \ref{alg:goodearl_basis} to compute the centralizers of the ODOs obtained.

\subsection{Hyperbolic coefficients}

Let us continue with Example \ref{exm:ideal:3:4:hyperbolic}
$$\bbL =\partial^3 + \frac{a_2}{\cosh(x)^2}\partial + \frac{a_3\sinh(x)}{\cosh(x)^3},$$
    with set of parameters $\Theta=\{a_1,a_2\}$.
 We run Algorithm~\ref{alg:get_ideal} for values of $M$ from $4$ up to $10$ observing the following: For $M \equiv 2 \pmod 3$, there are no points $\zeta\in \coC^2$ such that $\bbL(\zeta)$ has level $M$. For $M=4,7,10$ we obtain:
\[\cX_4\backslash \cX_2 = \{(-6, 0), (-6, 12)\},\quad \cX_7\backslash \cX_4=\{(-18, -12), (-18, 48)\},\]
\[\cX_{10}\backslash \cX_7=\{(-36, -48), (-36, 120)\}.\]
\begin{table}[t]
	\centering
	$\begin{array}{|c|c|c|c|}
		\hline
		\text{$M$} & \text{$|\bbL_M|$} & \text{Orders} & \text{Found Relations} \\
		\hline
		4 & 2 & (4, 5) & \text{None} \\
		7 & 2 & (7, 8) & \text{None} \\
		{10} & 2 & (10, 11) & \text{None}  \\
		\hline
	\end{array}$
    \label{tab:hyper:3}
    \caption{Centralizers for $n=3$ and variable $M$ on the hyperbolic case.}
\end{table}

For each of these points, we computed  centralizers using Algorithm~\ref{alg:goodearl_basis}. Table~\ref{tab:hyper:3} summarizes the results. The column Orders indicates the orders of the ODOs in the filtered basis $\{1,G^*_1,G^*_2\}$ of the centralizer $\centr (\bbL(\zeta))$, for $\bbL(\zeta)\in \bbL_M$. For instance if $M=10$, given $L=\bbL(-36,-48)$ or $L=\bbL(-36,120)$ then
\[\centr(L)=\bbC[L,G^*_1,G^*_2],\]
where $\ord(G^*_1)=10$ and $\ord(G^*_2)=11$. To access the explicit values of $G^*_1$ and $G^*_2$ see Section \ref{sec:implementation}. The last column Relations Found indicates that no algebraic relations where found between $G^*_1$ and $G^*_2$. Observe that because of their order values no algebraic relations can exist.

\subsection{Elliptic coefficients}

Consider the differential field $\coC \langle \wp(x)\rangle$, with derivation $d/dx$, which fits into the setting of differential hyperelliptic extensions described in Section~\ref{sec:computable}. In this case, we have $\eta = \wp(x)$ and then $\eta'$ is algebraic over $\bbC(\eta)$ since
\[\wp'(x)^2 = 4\wp(x)^3 - g_2\wp(x) - g_3,\,\,\, g_2, g_3\in \bbC.\]
We study operators
$$\bbL = \partial^3 + a_2\eta \partial + a_3 \eta'$$
and compute the points $\zeta=(a_2,a_3)$ in $\bbC^2$ such that $\bbL(\zeta)$ has a non-trivial centralizer of level $M=4$. We run Algorithm~\ref{alg:get_ideal}, obtaining three points:
\[\cX_4\backslash \cX_2 = \mathcal{F}_g \cup \mathcal{F}_0=\left\{ \left(-\frac{3}{2}, 0\right), \left(-\frac{3}{2}, -\frac{3}{2}\right)\right\} \cup  \left\{\left(-\frac{15}{4}, -\frac{15}{8}\right)\right\}.\]
It is important to remark that for $\zeta$ in $\mathcal{F}_0$, $\bbL(\zeta)$ has a non-trivial centralizer and level $4$ in the case $g_2 = 0$, for $\zeta$ in $\mathcal{F}_g$, $g_2$ and $g_3$ are arbitrary.

\begin{table}[h]
	\centering
	$\begin{array}{|c|c|c|c|}
		\hline
		\text{Family} & \text{\# Cases} & \text{Orders} & \text{Found Relations} \\
		\hline
		\mathcal{F}_g & 2 & (4, 5) &  \text{None}\\
		\mathcal{F}_0 & 1 & (4, 8) & G_2^* - G_1^{*2} \\
		\hline
	\end{array}$
    \label{tab:ell:3:4}
    \caption{Structure of the centralizers for $n=3$, $M=7$ on the elliptic case.}
\end{table}

We analyzed the centralizers in all the three cases using Algorithm~\ref{alg:goodearl_basis}. The results are summarized in Table~\ref{tab:ell:3:4}.
As an example, let us choose
$$L=\bbL\left(-\frac{3}{2}, -\frac{3}{2}\right) = \partial^3 - \frac{3}{2}\wp(x)\partial - \frac{3}{2}\wp'(x).$$ In this case, the centralizer is $\centr (L)=\coC[L,G^*_1,G^*_2]$ where
\begin{align*}
    &G^*_1 = -\left(\frac{3}{2} \wp(x)^{2} + \frac{2}{3} g_{2}\right) - \left(3 \wp'(x)\right)\partial - \left(2 \wp(x)\right)\partial^{2} + \partial^{4},\\
    &G^*_2 = -\left(\frac{5}{2} \wp(x) \wp'(x)\right) - \left(5 \wp(x)^{2} + 2 g_{2}\right)\partial - \left(5 \wp'(x)\right)\partial^{2} - \left(\frac{5}{2} \wp(x)\right)\partial^{3} + \partial^{5}.
\end{align*}
In the case
$L=\bbL\left(-\frac{15}{4}, -\frac{15}{8}\right)$, with $g_2=0$, the centralizer is $\centr (L)=\coC[L,G^*_1]$ since one relation was found $G^*_2=G_1^{*2}$.

\subsection{Rational coefficients}

To finish we will generate families of algebro-geometric ODOs with coefficients in the differential field
 $\bbC(x)$, with derivation $\partial=d/dx$. We apply Algorithm~\ref{alg:get_ideal} to an operator $\bbL$ of prime order $n$, fixing a set of algebraic parameters $\Theta=\{a_2,\ldots a_n\}$
\begin{equation}\label{eq-ratfam}
    \bbL= \partial^n + f_2\partial^{n-2}+\ldots+f_n,\,\,\, \mbox{ with } f_j=\frac{a_j}{x^j}.
\end{equation}

We run Algorithm~\ref{alg:get_ideal} for different values of $M$ to obtain families $\bbL_M$ of algebro-geometric ODOs with level $M$.  We analyze the level varieties $\cX_M$ of solutions of the parameters $\Theta$. In this case, of rational coefficients \eqref{eq-ratfam} with $(n,M)=1$, we observe that all the level varieties are finite sets of points.

As we saw in Example~\ref{exm:ideal:4:6:rational}, the system matrix $S_{\Theta}(n,m)$ is always of size $(|\Theta||J_{n,m}|)\times|J_{n,m}|$, built with $|\Theta|$ diagonal blocks of size $|J_{n,m}|\times|J_{n,m}|$. Moreover, we can build the matrix $S_{\Theta}(n,m+1)$ by adding $|\Theta|$ rows and $1$ column to $S_{\Theta}(n,m)$.
    The vector of constants that generates a non-trivial element in the centralizer of order $m$ using the basis $\cB(L)$ (see~\eqref{eq:wilson_basis} or \eqref{eq-ACbasis}) is always $\bc_m=(0,\ldots ,0, 1)$.

\begin{enumerate}
\item {\bf ($n=3$, $M=7$)} We run Algorithm \ref{alg:get_ideal} with input $\Theta=\{a_2,a_3\}$, $\{f_2= \frac{a_2}{x^2},f_3=\frac{a_3}{x^3}\}$ and $M=7$. The output is
    \[\cX_7 = \cX_6 \cup \mathcal{F}_1 \cup \mathcal{F}_2 \cup \mathcal{F}_3,\]
    where \[\mathcal{F}_1 = \{(-18, -12), (-18, 48)\},\quad\mathcal{F}_2=\{(-30, 0), (-30, 60)\},\quad\mathcal{F}_3 = \{(-48, 48)\}.\]
The output of Algorithm \ref{alg:get_ideal} with  $M=6$ is $\cX_6=\cX_5$ with
\[
\begin{cases}
\cX_1=\{(0,0)\},\\
\cX_2=\cX_1\cup\{(-3,3)\},\\
\cX_4=\cX_2\cup\{(-6,12),(-15,15),(-6,0)\},\\
\cX_5=\cX_4\cup\{(-24,24),(-12,24),(-12,0)\}.
\end{cases}
\]
    \begin{table}[b]
	\centering
	$\begin{array}{|c|c|c|c|}
                \hline
                \text{Family} & \text{\# Cases} & \text{Orders} & \text{Found Relations} \\
                \hline
                \mathcal{F}_1 & 2 & (7, 8) &  \text{None}\\
                \mathcal{F}_2 & 2 & (7, 11) &  \text{None}\\
                \mathcal{F}_3 & 1 & (7, 14) & G_2^* - G_1^{ *2} \\
                \hline
        \end{array}$
    \label{tab:rat:3:7}
    \caption{Structure of the centralizers for $n=3$, $M=7$ on the rational case.}
\end{table}

    For each of the points in $\cX_7\backslash \cX_5$, we computed the filtered basis of the centralizer with Algorithm \ref{alg:goodearl_basis}. Table~\ref{tab:rat:3:7} summarizes the results.

    Consider the  algebro-geometric operator
    $$L=\bbL(-48,48) = \partial^3 -\frac{48}{x^2}\partial + \frac{48}{x^3}.$$ The  filtered basis of the centralizer $\{1, G^*_1,G^*_2\}$ contains operators of orders $7$ and $14$. In this case the centralizer is $\centr (L)=\bbC[L,G^*_1]$, with
\begin{itemize}
    \item $G^*_1 = \left(\frac{24640}{x^{6}}\right)\partial + \left(\frac{-13440}{x^{5}}\right)\partial^{2} + \left(\frac{1120}{x^{4}}\right)\partial^{3} + \left(\frac{560}{x^{3}}\right)\partial^{4} + \left(\frac{-112}{x^{2}}\right)\partial^5 + \partial^7,$
    \item $G^*_2 = (G^*_1)^2$.
\end{itemize}

    Similarly, we computed the centralizer of $$L=\bbL(-18,-12)=\partial^3 - \frac{18}{x^2}\partial - \frac{12}{x^3}$$
   and obtained $\centr (L)=\bbC[L,G^*_1,G^*_2]$ with
    \begin{itemize}
        \item $G^*_1 = -\left(\frac{1120}{x^{7}}\right) + \left(\frac{1120}{x^{6}}\right)\partial - \left(\frac{560}{x^{5}}\right)\partial^{2} + \left(\frac{140}{x^{3}}\right)\partial^{4} - \left(\frac{42}{x^{2}}\right)\partial^{5} + \partial^{7},$
        \item $G^*_2 = \left(\frac{7040}{x^{8}}\right) - \left(\frac{7040}{x^{7}}\right)\partial + \left(\frac{3520}{x^{6}}\right)\partial^{2} - \left(\frac{640}{x^{5}}\right)\partial^{3} - \left(\frac{240}{x^{4}}\right)\partial^{4} + \left(\frac{208}{x^{3}}\right)\partial^{5} - \left(\frac{48}{x^{2}}\right)\partial^{6} + \partial^{8}$.
    \end{itemize}

    \item {\bf ($n=5$, $M=6$)} We run Algorithm \ref{alg:get_ideal} with input $\Theta=\{a_1,a_2\}$, $\{f_j=\frac{a_j}{x^j}\mid j=2,\ldots , 5\}$ and $M=6$. The output is
    \[\cX_6 = \cX_4 \cup \left(\bigcup_{i=1}^{13}\mathcal{F}_i\right),\]
    where $|\cX_6\setminus \cX_4| = 30$
    distributed in $13$ families of points gathered according to the structure of the centralizer. We computed the centralizers for each of the operators in
    \[\bbL_6=\{\bbL(\zeta)\mid \zeta\in \cX_6\backslash\cX_4\}.\]
    The results are summarized in Table~\ref{tab:rat:5:6}.
    \[\begin{array}{rcl}
        \mathcal{F}_1 & = & \left\{(-20, 0, 0, 0), (-25, 55, 70, -70), (-30, 60, 180, 0),\right. \\ & & \left.(-25, 95, -50, -270), (-30, 120, 0, -720), (-20, 120, -360, 480)\right\}\\
        \mathcal{F}_2 & = & \left\{(-35, 35, 0, 0), (-35, 175, -420, 420)\right\}\\
        \mathcal{F}_3 & = & \left\{(-40, 80, -80, 0), (-40, 160, -320, 320)\right\}\\
        \mathcal{F}_4 & = & \left\{(-60, 240, 0, 0), (-60, 120, 360, -1440)\right\}\\
        \mathcal{F}_5 & = & \left\{(-55, 245, -245, 0), (-55, 85, 235, -640)\right\}\\
        \mathcal{F}_6 & = & \left\{(-50, 40, 240, -240), (-50, 260, -420, 0), (-55, 175, 280, -280), \right.\\ & &\left.(-55, 155, 340, -1620)\right\}\\
        \mathcal{F}_7 & = & \left\{(-45, 135, -270, 270)\right\}\\
        \mathcal{F}_8 & = & \left\{(-65, 15, 630, 330), (-65, 375, -450, -1470), (-80, 240, 1200, -2880)\right\}\\
        \mathcal{F}_9 & = & \left\{(-100, 200, 880, 0), (-100, 400, 280, -3520)\right\}\\
        \mathcal{F}_{10} & = & \left\{(-85, 255, 130, -770)\right\}\\
        \mathcal{F}_{11} & = & \left\{(-95, 125, 1155, 0), (-95, 445, 195, -3840)\right\}\\
        \mathcal{F}_{12} & = & \left\{(-125, 175, 2730, -2730), (-125, 575, 1530, -7290)\right\}\\
        \mathcal{F}_{13} & = & \left\{(-175, 525, 3955, -8960)\right\}
    \end{array}\]

    \begin{table}[ht]
	\centering
	$\begin{array}{|c|c|c|c|}
                \hline
                \text{Family} & \text{\# Cases} & \text{Orders} & \text{Found Relations} \\
                \hline
                \mathcal{F}_1 & 6 & (6, 7, 8, 9) &\text{None} \\
                \mathcal{F}_2 & 2 & (6, 8, 9, 12) & G_4^* - G_1^{ *2} \\
                \mathcal{F}_3 & 2 & (6, 7, 9, 13) & G_4^* - G_1^{ *}G_2^{ *} \\
                \mathcal{F}_4 & 2 & (6, 7, 13, 14) & G_3^* - G_1^{ *}G_2^{ *}, G_4^* - G_2^{ *2} \\
                \mathcal{F}_5 & 2 & (6, 8, 12, 14) & G_3^* - G_1^{ *2}, G_4^* - G_1^{ *}G_2^{ *} \\
                \mathcal{F}_6 & 4 & (6, 9, 12, 13) & G_3^* - G_1^{ *2} \\
                \mathcal{F}_7 & 1 & (6, 7, 8, 14) & G_4^* - G_1^{ *}G_3^{ *} \\
                \mathcal{F}_8 & 3 & (6, 12, 13, 14) & G_2^* - G_1^{ *2} \\
                \mathcal{F}_9 & 2 & (6, 12, 13, 19) & G_2^* - G_1^{ *2}, G_4^* - G_1^{ *}G_3^{ *} \\
                \mathcal{F}_{10} & 1 & (6, 9, 12, 18) & G_3^* - G_1^{ *2}, G_4^* - G_1^{ *3} \\
                \mathcal{F}_{11} & 2 & (6, 12, 14, 18) & G_2^* - G_1^{ *2}, G_4^* - G_1^{ *3} \\
                \mathcal{F}_{12} & 2 & (6, 12, 18, 19) & G_2^* - G_1^{ *2}, G_3^* - G_1^{ *3} \\
                \mathcal{F}_{13} & 1 & (6, 12, 18, 24) & G_2^* - G_1^{ *2}, G_3^* - G_1^{ *3}, G_4^* - G_1^{ *4} \\
                \hline
        \end{array}$
    \label{tab:rat:5:6}
    \caption{Structure of the centralizers for $n=5$, $M=6$ on the rational case.}
\end{table}

     As an example, let us consider a point $\zeta$ of  $\mathcal{F}_5$, corresponding to the algebro-geometric ODO
    \[L=\bbL(-55,85,235,-640) = \partial^5 - \frac{55}{x^2}\partial^3 + \frac{85}{x^3}\partial^2 + \frac{235}{x^4}\partial - \frac{640}{x^5}.\]
    In this case, since $n=5$, the filtered $\bbC[L]$-basis is $\{1, G^*_1, G^*_2, G^*_3, G^*_4\}$ with orders $6, 8, 12, 14$ respectively. For details
    \begin{itemize}
        \item $G^*_1 = \left(\frac{1584}{x^{6}}\right) + \left(\frac{-1584}{x^{5}}\right)\partial + \left(\frac{360}{x^{4}}\right)\partial^{2} + \left(\frac{168}{x^{3}}\right)\partial^{3} + \left(\frac{-66}{x^{2}}\right)\partial^{4} + \partial^{6}$,
        \item $G^*_2 = \left(\frac{34944}{x^{8}}\right) + \left(\frac{-34944}{x^{7}}\right)\partial + \left(\frac{17472}{x^{6}}\right)\partial^{2} + \left(\frac{-5824}{x^{5}}\right)\partial^{3} + \left(\frac{160}{x^{4}}\right)\partial^{4} + \left(\frac{400}{x^{3}}\right)\partial^{5} + \left(\frac{-88}{x^{2}}\right)\partial^{6} + \partial^{8}$,
        \item $G^*_3 = (G^*_1)^2$,
        \item $G^*_4 = G^*_1 G^*_2$.
    \end{itemize}
The centralizer is $\centr (L)=\bbC[L,G^*_1,G^*_2]$.
\end{enumerate}

\section{Implementation and experimentation}\label{sec:implementation}

The results presented in this paper have been implemented in the SageMath package \texttt{dalgebra}. This package provides tools for working with differential algebra, including computations related to hyperelliptic differential extensions (see Definition~\ref{def-elliptic}). The package is still under active development, and we provide the implementation of all algorithms starting from version 0.0.7, available at:
\begin{center}
    \url{https://github.com/Antonio-JP/dalgebra/tree/v0.0.7}
\end{center}

In order to ensure reproducibility and verification of our theoretical results, we provide a Jupyter Notebook in the repository, where we have detailed the necessary code to execute the examples in this paper. This notebook contains a section for each example including the corresponding code. Besides reproducing the examples on this paper, this notebook also provides code examples to use of the package \texttt{dalgebra}. The notebook, called \emph{EffectiveCentralizers\_PaperExamples.ipynb}, can be found in the following link:
\begin{center}
    \url{https://github.com/Antonio-JP/dalgebra/blob/v0.0.7/notebooks/EffectiveCentralizers_PaperExamples.ipynb}
\end{center}

Recall that in Section~\ref{sec:families} we presented a systematic set of computations for linear differential operators with rational, hyperbolic and elliptic coefficients. These results were computed using the aforementioned software package \texttt{dalgebra} (version 0.0.7). The code to reproduce these examples can be found in the folder \emph{experiments/centralizers}. There, the file \texttt{script.sage} can be found, with the SageMath code that takes the values $n$, $m$, and a family of operators (namely, \texttt{"rational"}, \texttt{"hyperbolic"} or \texttt{"elliptic"}) and computes all the possible operators of order $n$ with a non-trivial centralizer of level $m$ for the given family.

While writing this paper, we have executed this script in SageMath~\cite{sagemath} on a computer with the following specifications:
\begin{itemize}
    \item Processor: Intel(R) Xeon(R) W-3323 CPU @ 3.50GHz
    \item RAM: 50GB
    \item Operating System: AlmaLinux 9.5 (Teal Serval)
    \item SageMath Version: 10.5
\end{itemize}

For each execution, four different files have been created in the subfolder \emph{experiments/centralizers/results}:
\begin{itemize}
    \item A file \texttt{<family>\_$n$\_$m$\_matrix.md}, with the matrix $S_\Theta(n,m)$. This file can be executed in Maple and SageMath to generate the system matrix in both Computer Algebra systems.
    \item A file \texttt{<family>\_$n$\_$m$\_rows.md}, contains a list of tuples $(k, t)$ where $k \in \mathbb{N}$ and $t$ are monomials. These tuples index each row in the \emph{matrix} file. Namely, if the $i$th tuple in the file is $(k,t)$, then the $i$th row of the matrix was the coefficient of monomial $t$ from the $k$-th equation in the system $\hgd_{n,m}(\bbL)$.
    \item A file \texttt{<family>\_$n$\_$m$\_report.md}, contains a summary of all the computations. More precisely, it lists and analyzes all the operators of order $n$ and level $m$ and their centralizers. It finishes by summarizing all cases splitting them into families regarding the orders of the ODOs in the filtered basis.
    \item A file \texttt{<family>\_$n$\_$m$\_table.tex}, contains the TeX code for the tables displayed in this paper. They were generated from the \emph{report} file.
\end{itemize}

\noindent{\bf Acknowledgements.} All authors are partially supported by the grant PID2021-124473NB-I00, ``Algorithmic Differential Algebra and Integrability" (ADAI)  from the Spanish MCIN/AEI /10.13039/501100011033 and by FEDER, UE.

\bigskip
\bibliographystyle{acm}
\bibliography{Bibliography}



\end{document}